\documentclass[reqno]{amsart}

\usepackage[english]{babel} 						
\usepackage[T1]{fontenc}							
\usepackage[utf8]{inputenx}

\usepackage{mathtools}								
\usepackage{empheq}									
\usepackage{enumitem}								

\usepackage{amssymb}								
\usepackage[sc]{mathpazo}							
\usepackage{mathrsfs}								

\usepackage{todonotes}								
\usepackage{aliascnt}								
\usepackage{comment}

\usepackage[normalem]{ulem}
\usepackage{epsfig}
\usepackage{a4wide,color,eucal,mathrsfs}
\usepackage[normalem]{ulem}
\usepackage{amsmath,amssymb,epsfig,amsthm} 

\usepackage{pgfplots}

\usepackage[babel]{csquotes}							

\usepackage{nameref}									
\usepackage[colorlinks	=	true,						
          	 	linkcolor	=	red,
            	urlcolor	=	red,
            	citecolor	=	red]%
            	{hyperref}
\usepackage{bookmark}									
\usepackage{cleveref}									

\newcommand{\apmd}[2][]{							
	\ifthenelse{\equal{#1}{}}%
					{ \operatorname{N}_{#2}	}%
					{ \operatorname{N}_{#1,#2} 	}}

\newcommand{\LIP}{{\rm LIP}}

\newcommand{\aint}[2][]{
	\ifthenelse{\equal{#1}{}}%
					{%
\mathchoice%
      {\mathop{\kern 0.2em\vrule width 0.6em height 0.69678ex depth -0.58065ex
              \kern -0.8em \intop}\nolimits_{\kern -0.45em#2}^{#1}}%
      {\mathop{\kern 0.1em\vrule width 0.5em height 0.69678ex depth -0.60387ex
              \kern -0.6em \intop}\nolimits_{#2}^{#1}}%
      {\mathop{\kern 0.1em\vrule width 0.5em height 0.69678ex depth -0.60387ex
              \kern -0.6em \intop}\nolimits_{#2}^{#1}}%
      {\mathop{\kern 0.1em\vrule width 0.5em height 0.69678ex depth -0.60387ex
              \kern -0.6em \intop}\nolimits_{#2}^{#1}}}%
					{%
\mathchoice%
      {\mathop{\kern 0.2em\vrule width 0.6em height 0.69678ex depth -0.58065ex
              \kern -0.8em \intop}\nolimits_{\kern -0.45em#1}^{#2}}%
      {\mathop{\kern 0.1em\vrule width 0.5em height 0.69678ex depth -0.60387ex
              \kern -0.6em \intop}\nolimits_{#1}^{#2}}%
      {\mathop{\kern 0.1em\vrule width 0.5em height 0.69678ex depth -0.60387ex
              \kern -0.6em \intop}\nolimits_{#1}^{#2}}%
      {\mathop{\kern 0.1em\vrule width 0.5em height 0.69678ex depth -0.60387ex
              \kern -0.6em \intop}\nolimits_{#1}^{#2}}}}

\numberwithin{equation}{section}

\newtheorem{theorem}{Theorem}[section]

\newtheorem{corollary}[theorem]{Corollary}
\newtheorem{lemma}[theorem]{Lemma}
\newtheorem{proposition}[theorem]{Proposition}

\newtheorem{definition}[theorem]{Definition}

\newtheorem{example}[theorem]{Example}
\newtheorem{question}[theorem]{Question}

\theoremstyle{remark}
\newtheorem{remark}[theorem]{Remark}

\DeclareMathOperator{\diam}{diam}

\newcommand{\R}{\mathbb{R}}

\def\fz{\infty}

\DeclareMathOperator{\lip}{Lip}

\definecolor{emerald}{rgb}{0.31, 0.78, 0.47}

\def\bint{{\ifinner\rlap{\bf\kern.35em--}
\int\else\rlap{\bf\kern.45em--}\int\fi}\ignorespaces}

\def\bbint{{\ifinner\rlap{\bf\kern.35em--}
\hspace{0.078cm}\int\else\rlap{\bf\kern.45em--}\int\fi}\ignorespaces}

\def\diam{{\mathop\mathrm{\,diam\,}}}

\def\dfrac{\displaystyle\frac}

\def\r{\right}
\def\lf{\left}

\def\bint{{\ifinner\rlap{\bf\kern.35em--}
\int\else\rlap{\bf\kern.45em--}\int\fi}\ignorespaces}

\begin{document}

\title[Thick quasiconvexity and applications]{On infinity thick quasiconvexity and applications}

\author{Miguel Garc\'ia-Bravo}
\address{Institute of Interdisciplinary Mathematics, Department of Mathematical Analysis and Applied
Mathematics, Faculty of Mathematics, Complutense University of Madrid, 28040 Madrid, Spain} \email{miguel05@ucm.es}

\author{Toni Ikonen} 
\address{Department of Mathematics, University of Fribourg, Chemin du Musée 23, 1700 Fribourg, Switzerland.}

\email{toni.ikonen@unifr.ch}

\author{Zheng Zhu}

\address{School of Mathematic Sciences\\
            Beihang University\\
             Beijing, 102206\\
            P. R. China}

\email{zhzhu@buaa.edu.cn}

\keywords{Sobolev space, Lipschitz representative, thick quasiconvexity, essential distance, $\infty$-harmonic, Sobolev extension domain}
\thanks{M.G-B. acknowledges the support provided by the grant PID2022-138758NB-I00
from the Ministerio de Ciencia e Innovación of Spain. T.I. was supported by the Swiss National Science Foundation grant 212867. Z.Z. was supported by the NSFC grant (No. 12301111) and “the Fundamental Research Funds for the Central Universities” in Beihang University and the Beijing Natural Science Foundation (No. 1242007).}
\subjclass[2020]{Primary: 46E36. Secondary: 30L99, 46E35}


\begin{abstract}
We investigate geometric characterizations of a metric measure space where every function in the Newton--Sobolev space $N^{1,\infty}(Z)$ has a Lipschitz representative. We prove that when the metric space is locally complete and the reference measure is infinitesimally doubling, the above property is equivalent to the space being very $\infty$-thick quasiconvex up to a scale. That is, up to some scale, every pair of points can be joined by a family of quasiconvex curves that is not negligible for the $\infty$-modulus.

As a first application, we prove a local-to-global improvement for the weak $(1,\infty)$-Poincaré inequality for locally complete quasiconvex metric spaces that have a doubling reference measure. As a second application, we apply our results to the existence and uniqueness of $\infty$-harmonic extensions with Lipschitz boundary data for precompact domains in a large class of metric measure spaces. As a final application, we illustrate that in the context of Sobolev extension sets, very $\infty$-thick quasiconvexity up to a scale plays an analogous role as local uniform quasiconvexity does in the Euclidean space.

Our assumptions are adapted to the analysis of Sobolev extension sets and thus avoid stronger assumptions such as the doubling property of the measure. Examples satisfying our assumptions naturally occur as simplicial complexes, GCBA spaces, and metric quotients of Euclidean spaces.
\end{abstract}
\maketitle

\section{Introduction}

In this paper, we consider locally complete metric measure spaces $Z$ equipped with a reference measure that is infinitesimally doubling, and finite and positive on all balls. 

We seek \emph{geometric} conditions on $Z$ equivalent to every Sobolev function in $W^{1,\infty}(Z)$ having a Lipschitz representative. There are several definitions of Sobolev spaces in metric measure spaces. Two widely used ones are the Haj{\l}asz--Sobolev space $M^{1,p}(Z)$ \cite{M=W} and the Newton--Sobolev space $N^{1,p}(Z)$ \cite{Shanmugalingam}. We refer to the recent surveys \cite{Amb:Col:DiMa:15,Amb:Iko:Luc:Pas:24,Amb:Iko:Luc:Pas:24:correction} and the monographs \cite{Bj:Bj:11,HKST2015,Sav:22} for overviews on the expanding field of Sobolev analysis on metric measure spaces.

Since we only consider metric measure spaces in which every ball has a positive and finite measure, the Haj{\l}asz--Sobolev space $M^{1,\infty}(Z)$ is isomorphic to the space of bounded Lipschitz functions $\LIP^{\infty}(Z)$, where the isomorphism is the standard embedding of $\LIP^{\infty}(Z)$ into $L^{\infty}(Z)$. There is also a uniquely defined map from $N^{1,\infty}(Z)$ into $L^{\infty}(Z)$, where every equivalence class from $N^{1,\infty}(Z)$, uniquely defined Sobolev $\infty$-capacity almost everywhere, is sent to its Lebesgue equivalence class in $L^{\infty}(Z)$. The image of this embedding will be our Sobolev space $W^{1,\infty}(Z)$. In spaces such as Euclidean spaces, Riemannian manifolds, or sub-Riemannian manifolds, the Sobolev space is isometrically isomorphic to the standard definition.

\subsection{Sobolev functions and Lipschitz representatives}
We seek geometric conditions equivalent to $\LIP^{\infty}(Z) = W^{1,\infty}(Z)$. When $Z$ is a Euclidean domain, the equality holds if and only if the domain is locally uniformly quasiconvex. In the metric measure space setting this is not enough. Indeed, while the standard Sierpi\'nski carpet, see \Cref{carpet:intro}, is quasiconvex and the natural self-similar measure is a non-trivial doubling measure, the Sobolev space $W^{1,\infty}(Z)$ coincides with $L^{\infty}(Z)$, which contains $\LIP^{\infty}(Z)$ as a proper subspace; see e.g. \cite[Section 9]{Bj:Bj:Mal:25}. Due to this striking example, a stronger form of connectivity is needed. The equality $\LIP^{\infty}(Z) = W^{1,\infty}(Z)$ is known when $Z$ is a locally complete $p$-PI, i.e. a metric measure space with a doubling measure and supporting a weak $(1,p)$-Poincaré inequality for some $p \in [1,\infty]$ in the sense of Heinonen--Koskela \cite{Hei:Ko:98}. The strongest of the Poincaré inequalities is the weak $(1,1)$-Poincaré inequality, which has several geometric characterizations, see e.g. the recent survey \cite{Caputo}, while the weakest one is the weak $(1,\infty)$-Poincaré inequality, introduced in \cite{Durand,Durand_2}. When the reference measure is doubling, the weak $(1,\infty)$-Poincaré inequality has the following characterization due to Durand-Cartagena, Jaramillo, and Shanmugalingam.

\begin{proposition}{\cite[Theorem 3.1]{D-CJS2016}}\label{prop_DURAND}
    Let $Z$ be a locally complete metric measure space with a doubling measure that is finite and positive on all balls. Then the following conditions are equivalent:
    \begin{enumerate}
        \item $Z$ supports a weak $(1,\infty)$-Poincaré inequality;
        \item $Z$ is very $\infty$-thick quasiconvex: there exists $C\geq 1$ so that every pair of points $x,y \in Z$ can be joined with a curve whose length is at most $C d(x,y)$ and the curve can be chosen to be in the complement of any given $\infty$-negligible family of curves;
        \item $Z$ is connected and $\LIP^{\infty}(Z)=W^{1,\infty}(Z)$ with comparable energy seminorms.
    \end{enumerate}
\end{proposition}
While the equality $\LIP^{\infty}(Z) = W^{1,\infty}(Z)$ does not imply that the energy seminorms are comparable, e.g. for a pair of separated Euclidean balls, it does imply that the norms of the two spaces are equivalent in the sense that there exists $c \in (0,1]$ for which
\begin{align}\label{eq:equivalenceofnorms}
    c\|u\|_{ \LIP^{\infty}(Z) }
    \leq
    \|u\|_{ W^{1,\infty}(Z) }
    \leq
    \|u\|_{ \LIP^{\infty}(Z) }
    \quad\text{for every $u \in W^{1,\infty}(Z)$.}
\end{align}
These standard norms are such that the map $P \colon \LIP^{\infty}(Z) \rightarrow W^{1,\infty}(Z), \,u \mapsto u$ is injective and has operator norm at most one. Thus, by the bounded inverse theorem, the surjectivity of $P$ implies the existence of a constant $c$ as in \eqref{eq:equivalenceofnorms}.

\begin{figure}
    \centering
    \includegraphics[width=0.25\linewidth]{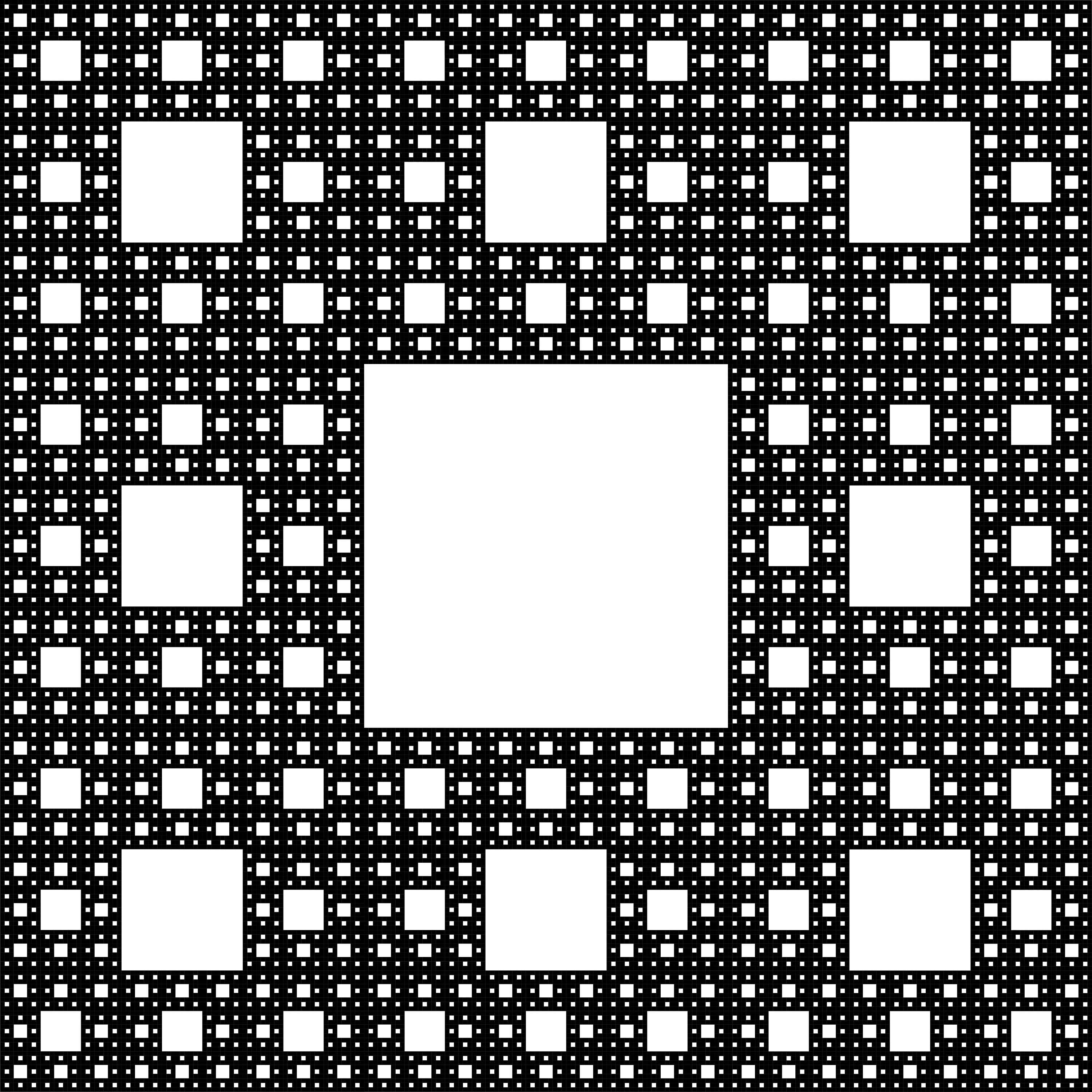}
    \caption{Sierpi\'nski carpet $S$}
    \label{carpet:intro}
\end{figure}

In contrast to the very $\infty$-thick quasiconvexity property defined in \Cref{prop_DURAND} (2), a geometric property equivalent to \eqref{eq:equivalenceofnorms} has to be \emph{local}. For example, if $Z$ is the union of two separated open Euclidean balls and the separation is decreased to zero, the corresponding constants $c$ converge to zero. This phenomenon is captured by our first theorem,  where we localize the very $\infty$-thick quasiconvexity from \cite{D-CJS2016}.
\begin{theorem}\label{thm:equivalence}
Let $Z$ be a locally complete metric measure space with a reference measure that is infinitesimally doubling, and finite and positive on all balls. Then the following are quantitatively equivalent.
\begin{enumerate}
    \item $\LIP^{\infty}(Z) = W^{1,\infty}(Z)$, where  \eqref{eq:equivalenceofnorms} holds with a constant $c > 0$;
    \item $Z$ is very $\infty$-thick $(C,R)$-quasiconvex for some $C \geq 1$ and $R \in (0,\infty]$: for every $x_0 \in Z$, every pair $x,y \in B( x_0, R )$ can be joined with a curve whose length is at most $C d(x,y)$ and the curve can be chosen to be in the complement of any given $\infty$-negligible family of curves.
\end{enumerate}
Either of the properties (1) or (2) imply that $Z$ supports a weak $(1,\infty)$-Poincaré inequality up to some scale. Conversely, if the reference measure is doubling up to some scale and the weak $(1,\infty)$-Poincaré inequality holds up to a scale, then (2) holds. Both of these statements are quantitative.
\end{theorem}
\Cref{thm:equivalence} (2) implies that, up to a scale, $Z$ is quantitatively well-connected; the dependence of $C$ and $R$ on $c$ is discussed in \Cref{rem:asymptoticbehaviour} and in the special case $c = 1$ in  \Cref{prop:sharpness} below, as well as in \Cref{sec:essentialdistance:proofs}.

One of the main novelties of our \Cref{thm:equivalence} compared to \Cref{prop_DURAND} is to consider infinitesimally doubling spaces instead of doubling ones. Our method of proof of \Cref{thm:equivalence} shows that the doubling condition is not necessary anywhere in the argument, and, in fact, the Lebesgue differentiation theorem is enough to derive geometric consequences of \eqref{eq:equivalenceofnorms}.

\begin{remark}\label{rem:potential:consequences}
\Cref{thm:equivalence} (2) has potential theoretic consequences. Indeed, it implies that every equivalence class in $N^{1,\infty}(Z)$ contains a unique locally Lipschitz function. The equivalence with \Cref{thm:equivalence} (1) yields that the unique representative is, in fact, Lipschitz. Since an equivalence class in $N^{1,\infty}(Z)$ is closed under modifications in a set negligible for Sobolev $\infty$-capacity, it follows that every point has positive Sobolev $\infty$-capacity.
\end{remark}

An interesting byproduct of \Cref{thm:equivalence} is the following result, where we characterize the situation when the spaces $\LIP^{\infty}(Z)$ and $ W^{1,\infty}(Z)$ are isometric. The proof is given in \Cref{rem:asymptoticbehaviour} and follows from a careful tracking of the dependence of the constants appearing in the proof of \Cref{thm:equivalence}.  
\begin{theorem}\label{prop:sharpness}
Let $Z$ be a locally complete metric measure space with a reference measure that is infinitesimally doubling, and finite and positive on all balls. Then $\LIP^{\infty}(Z) = W^{1,\infty}(Z)$ isometrically if and only if $Z$ is very $\infty$-thick $C$-quasiconvex for every $C>1$.
\end{theorem}

We emphasize that metric measure spaces whose connected components accumulate cannot satisfy the conclusion of \Cref{thm:equivalence} (2), such as the complementary components of the standard Sierpi\'nski carpet in \Cref{carpet:intro}. On the other hand, when $Z$ is quasiconvex, more can be said. This is our  third theorem.

\begin{theorem}\label{thm:bounded:doubling}
Let $Z$ be a locally complete metric measure space with a reference measure that is infinitesimally doubling, and finite and positive on all balls. Then the following are quantitatively equivalent:
\begin{enumerate}
    \item $\LIP^{\infty}(Z) = W^{1,\infty}(Z)$ and $Z$ is quasiconvex;
    \item $Z$ is very $\infty$-thick quasiconvex;
    \item $\LIP^{\infty}(Z) = W^{1,\infty}(Z)$ with comparable energy seminorms.
\end{enumerate}
Under any one of these assumptions, $Z$ supports a weak $(1,\infty)$-Poincaré inequality.
\end{theorem}
Observe, in particular, that \Cref{thm:bounded:doubling} implies that the connectivity assumption in \Cref{prop_DURAND} (3) is a consequence of the energy seminorm comparability. Thus \Cref{thm:bounded:doubling} leads to a strengthening of \Cref{prop_DURAND}.

We emphasize that the quasiconvexity assumption in \Cref{thm:bounded:doubling} (1) cannot be relaxed to connectivity. Indeed, the domain
\begin{equation}\label{eq:nonquasiconvexdomain}
    \Omega=\R^2 \setminus (-\infty,0] \times [-1,1]
\end{equation}
serves as a counterexample. The closure $\overline{\Omega}$ illustrates that the quasiconvexity is needed even in the complete case. 

A combination of \Cref{thm:equivalence,thm:bounded:doubling} also leads to a local-to-global phenomenom of the weak $(1,\infty)$-Poincaré inequality for quasiconvex metric spaces equipped with a doubling measure. To the best of our knowledge, the result is new.
\begin{corollary}\label{cor:bounded:doubling:PI}
Let $Z$ be a quasiconvex locally complete metric measure space such that the reference measure is doubling, and finite and positive on all balls. Then $Z$ supports a weak $(1,\infty)$-Poincaré inequality if and only if $Z$ supports a weak $(1,\infty)$-Poincaré inequality up to some scale.
\end{corollary}
Again the domain \eqref{eq:nonquasiconvexdomain} and its closure illustrate that quasiconvexity of $Z$ is necessary for the self-improvement to hold.

Next, we describe two applications of our methods.
\subsection{Infinity harmonic functions}
We now move the discussion towards $\infty$-harmonic functions. To this end, given an open set $\Omega \subset Z$, we say that $u \in W^{1,\infty}(Z)$ is \emph{$\infty$-harmonic} in $\Omega$ if whenever $V \subset \Omega$ is an open set and $v \in W^{1,\infty}(Z)$ is such that $u = v$ in $Z \setminus V$, it holds that
\begin{align}\label{eq:localminimalenergy}
    \mathscr{E}_{\infty}( u|_{V} ) \leq \mathscr{E}_{\infty}( v|_{V} )
\end{align}
where $\mathscr{E}_{\infty}(h)$ refers to the energy of $h \in W^{1,\infty}(V)$ as defined in \eqref{def:N1P}. For our purposes, $\infty$-harmonic functions will only be considered whenever $W^{1,\infty}(Z)=\LIP^{\infty}(Z)$ hence we may use Lipschitz representatives when considering the comparison in \eqref{eq:localminimalenergy}. Thus, the boundary value problem is well-posed whenever $\Omega$ has a non-empty boundary. Our techniques lead to the following existence result.
\begin{theorem}\label{thm:infty-harmonic-solution}
Let $Z$ be a complete and locally compact metric measure space with a reference measure that is infinitesimally doubling, and finite and positive on all balls. Suppose that $\LIP^{\infty}(Z) = W^{1,\infty}(Z)$ and $\Omega \subset Z$ is a precompact domain such that $\overline{\Omega} \setminus \Omega \neq \emptyset$. Then, for every  $g \in \LIP^{\infty}(Z)$, there exists a unique function $u \in \LIP^{\infty}(Z)$ such that $u=g$ on $ Z \setminus \Omega$ and $u$ is  $\infty$-harmonic in $\Omega$.
\end{theorem}

In Euclidean spaces, it is a deep theorem that  a Lipschitz function $u \colon \R^n\to\R$ being $\infty$-harmonic in some open set $V\subset\R^n$ is equivalent to being an AMLE function of $u|_{\R^n\setminus V}$. AMLEs were introduced by Aronsson \cite{Aron:67}, and we refer the reader to the monograph \cite{Aron:Cran:Juu:04} for an exposition of the Euclidean theory. AMLEs were later adapted to length spaces in \cite{Juut:02,Per:Schr:Sheff:Wil:09} in the following way: Given a non-empty closed set $A \subset Z$ and a Lipschitz function $g \colon A \rightarrow \mathbb{R}$, a Lipschitz extension $u \colon Z \rightarrow \mathbb{R}$ is an \emph{absolutely minimizing Lipschitz extension} (AMLE) of $g$ if $g|_{A} = u|_{A}$ and the Lipschitz constants satisfy $\lip( u|_{V} ) = \lip( u|_{\partial V} )$ whenever $V \subset Z \setminus A$ is a non-empty open set. In a length space, a standard argument shows that $u \colon Z \rightarrow \mathbb{R}$ is an AMLE of $g$ if and only if $u|_{ A } = g|_{ A }$ and $\lip( u|_{V} ) \leq \lip( v|_{V} )$ whenever $v \colon \overline{V} \rightarrow \mathbb{R}$ is a Lipschitz extension of $u|_{ \partial V } \colon \partial V \rightarrow \mathbb{R}$ and $V \subset Z \setminus A$ is a non-empty open set.

The $\infty$-harmonicity in the formulation above was introduced by Juutinen and Shanmugalingam in \cite{Juut:Shan:06} where the authors considered connections between $\infty$-harmonicity and AMLEs in metric measure spaces. More recently, the connection was further refined by Durand-Cartagena, Jaramillo, and Shanmugalingam in \cite{Dur:Jar:Sha:19}. In $\infty$-PI spaces, the authors introduced the \emph{essential distance} as a way to construct a bi-Lipschitz comparable length distance such that $\infty$-harmonic functions in the old distance are transformed to AMLEs in the new distance and vice versa. \Cref{thm:equivalence} allows the adaptation of their strategy in greater generality, thereby leading to \Cref{thm:infty-harmonic-solution}.

%
%
\subsection{Sobolev extension sets}
In this section, we apply \Cref{thm:equivalence} to \emph{Sobolev extension sets}. To set the stage, we recall the following theorem by Haj{\l}asz, Koskela, and Tuominen.

\begin{proposition}{\cite[Theorem 7]{HKT2008}}\label{prop:HKT2008}
    Let $\Omega\subset\R^n$ be a domain. Then the following conditions are equivalent:
    \begin{enumerate}
        \item for every $u\in W^{1,\infty}(\Omega)$, there exists $v\in W^{1,\infty}(\R^n)$ such that $v|_{\Omega}=u$;
        \item $\Omega$ is $(C,R)$-quasiconvex for some $C \geq 1$ and $R > 0$: for every $x_0 \in \Omega$, every $x, y \in \Omega \cap B( x_0, R )$ can be joined with a $C$-quasiconvex curve $\gamma \colon [0,1] \rightarrow \Omega$.
  \end{enumerate}
\end{proposition}

A domain (resp. set) satisfying \Cref{prop:HKT2008} (1) is called a \emph{$W^{1,\infty}$-extension domain (resp. set)}.  Note that, in this case, the restriction operator $R:W^{1,\infty}(\R^n)\to W^{1,\infty}(\Omega)$, defined as $R(v)=v|_{\Omega}$, is surjective. Hence, by the bounded inverse theorem, there exists a constant $C>0$ such that any $u\in W^{1,\infty}(\Omega)$ admits an extension $v\in W^{1,\infty}(\R^n)$ with $\|v\|_{W^{1,\infty}(\R^n)}\leq C\|u\|_{W^{1,\infty}(\Omega)}$. See also \Cref{def:sob_ext_set}.

We formulate the following metric version of \Cref{prop:HKT2008}.
\begin{theorem}\label{thm:Main_inf.ext.}
Let $Z$ be a locally complete metric measure space with a reference measure that is infinitesimally doubling, and finite and positive on all balls. Suppose that $Z$ is very $\infty$-thick $(C,R)$-quasiconvex for some $C\geq 1$, $R>0$ and $\Omega \subset Z$ is open. Then the following are equivalent.
\begin{enumerate}
    \item $\Omega$ is a $W^{1,\infty}$-extension set;
    \item $\Omega$ is very $\infty$-thick $( C', R' )$-quasiconvex for some $C' \geq 1, R' > 0$;
    \item $\Omega$ is $(C'',R'')$-quasiconvex for some $C''\geq 1$, $R''>0$.
\end{enumerate}
In fact, if (3) holds, then (2) holds for $C' = CC''$ and $R' = R''$.
\end{theorem}
Under the assumptions of \Cref{thm:Main_inf.ext.}, being a $W^{1,\infty}$-extension set is equivalent to the equality $\LIP^{\infty}(\Omega) = W^{1,\infty}(\Omega)$ because every element of $\LIP^{\infty}(\Omega)$ can be extended to $\LIP^{\infty}(Z)$ using standard arguments.

While \Cref{thm:Main_inf.ext.,thm:equivalence} imply that a $W^{1,\infty}$-extension domain supports a weak $\infty$-Poincaré inequality up to some scale, it is not clear if the converse is true; see \Cref{ques:PI}. Indeed, the restriction of $\mu$ to $\Omega$ does not have to be doubling up to some scale in general as follows e.g. from considering certain cusps as in \Cref{ex:2} (see also \cite{BKZheng}). The failure of the doubling property is in contrast to $W^{1,p}$-extension domains when $p \in [1,\infty)$, cf. \cite{HKT2008:B,HKT2008}.

\Cref{thm:Main_inf.ext.} applies to a large class of domains in $p$-PI spaces for $p \in [1,\infty]$. Indeed, every bounded domain in a PI space can be approximated from inside and outside by uniform domains which are, in particular, quasiconvex domains, cf. \cite{Raj:21}. Thus \Cref{thm:Main_inf.ext.} (3) holds for such domains.

In the recent work \cite{Garc:Iko:Zhu:23}, we considered Banach-valued $W^{1,p}$-extension sets for exponents $p \in (1,\infty)$. More specifically, we were interested in determining whether a domain $\Omega \subset \mathbb{R}^n$ is a Banach-valued $W^{1,p}$-extension domain if and only if it is a real-valued $W^{1,p}$-extension domain, see e.g. \cite{Kos:99,Bjo:Sha:07,HKT2008:B} for background on the topic. We verified the equivalence for $p \in [n,\infty)$ while the case $p \in [1,n)$ remains open, cf. \cite[Question 1.1]{Garc:Iko:Zhu:23}. The case $p = \infty$ is simpler and holds under mild assumptions.
\begin{proposition}\label{prop:Banach-valued_version}
Let $Z$ be a locally compact metric measure space with a reference measure that is finite and positive on all balls, and that $\LIP^{\infty}(Z) = W^{1,\infty}(Z)$. Then the following are equivalent for a locally complete $\Omega \subset Z$ for which $\Omega$ or $Z \setminus \Omega$ has finite Nagata dimension:
\begin{enumerate}
 \item $\Omega$ is a real-valued $W^{1,\infty}-$extension set.
 \item $\Omega$ is a $\mathbb V$-valued $W^{1,\infty}-$extension set for every Banach space $\mathbb{V}$.
 \end{enumerate}
Moreover, there exists a linear extension operator in the cases $(1)$ and $(2)$.
\end{proposition}
The Nagata dimension assumption implies that $Z \setminus \overline{\Omega}$ has a Whitney-type decomposition with bounded multiplicity, cf. \cite[Section 5]{La:Sch:05}, which allows the construction of the extension operator by standard methods. We recall that every metric space which is metrically doubling, or, in particular, admits a doubling measure, has finite Nagata dimension \cite[Lemma 2.3]{La:Sch:05}. See \cite{La:Sch:05} for further information on Nagata dimension.

\subsection{Structure of the paper}
The paper is structured as follows. In \Cref{sec:prelim}, we present the definitions and main properties of the space of bounded Lipschitz functions $\LIP^{\infty}(Z)$, the Haj{\l}asz--Sobolev space $M^{1,\infty}(Z)$, the Newton--Sobolev space $N^{1,\infty}(Z)$, and the Sobolev space $W^{1,\infty}(Z)$. In \Cref{sec:thick}, we prove \Cref{thm:equivalence,thm:bounded:doubling,prop:sharpness}, and also \Cref{cor:bounded:doubling:PI}. In fact, in \Cref{prop:thick:to:Hajlasz,prop:Hajlasz:to:thick,prop:quasiconvexity_2}, we consider in detail the relations between the equality $\LIP^{\infty}(Z)=W^{1,\infty}(Z)$, the $\infty$-thick quasiconvexity, and the weak $(1,\infty)$-Poincaré inequality. We construct the essential distance and provide the proof of \Cref{thm:infty-harmonic-solution} in \Cref{sec:essentialdistance:proofs}. Results about real-valued and Banach-valued Sobolev extension sets, in particular, \Cref{thm:Main_inf.ext.} and \Cref{prop:Banach-valued_version},  appear in \Cref{sec:extension,sec:banach}, respectively. Finally, we construct several examples in \Cref{sec:examples} satisfying the conclusion of \Cref{thm:equivalence} but which are not PI spaces.

\section{Preliminaries}\label{sec:prelim}

A \emph{metric measure space} is a triple $( Z, d_Z, \mu_{Z} )$, where $(Z,d_Z)$ is a metric space and $\mu_{Z}$ is a Borel regular measure on $Z$ with $0 < \mu_{Z}( B(z,r) ) < \infty$ for every $z \in Z$ and $0<r<\fz$. We sometimes simply write $Z$ for the triple $( Z, d_Z, \mu_Z )$. In particular, we always assume that $(Z,d_Z)$ is separable. When there is no chance for confusion, we omit the subscripts from $d_{Z}$ and $\mu_{Z}$.  

We say that $N \subset Z$ is \emph{negligible} if $\mu_Z( N ) = 0$. A property holds \emph{almost everywhere} in $F \subset Z$,  or for almost every $z \in F$, if there exists a negligible set $N \subset F$ such that the property holds for every point in $F \setminus N$.

We denote the \emph{closed} and \emph{open} balls of $Z$, respectively, by $$\overline{B}_Z(z,r) \coloneqq \left\{ y \in Z \colon d(z,y) \leq r \right\};\quad \quad B_Z(z,r) \coloneqq \left\{ y \in Z \colon d(z,y) < r \right\}.$$ We sometimes drop the subscript $Z$ and also $(z,r)$, and use the notation $\lambda B$ to denote the ball $B( z, \lambda r )$  for $0<\lambda<\fz$.

In what follows, unless otherwise stated, any measurable subset of $\R^n$ is seen as a metric measure space with the Euclidean distance and the restricted Lebesgue measure. Similarly, for a given measurable subset $F\subset Z$ of a metric measure space $Z$ we may consider $F$ itself as a metric measure space with the restricted measure and distance coming from $Z$.

Above and below, we say that a metric measure space $Z$ is \emph{infinitesimally doubling} if
\begin{equation*}
    \limsup_{ r \rightarrow 0^{+} } \frac{ \mu( B(z,2r) ) }{ \mu( B(z,r) ) } < \infty
    \quad\text{for $\mu$-almost every $z \in Z$}.
\end{equation*}
We say that $\Omega \subset Z$ is \emph{infinitesimally doubling} if $B(z,2r)$ and $B(z,r)$ can be replaced by $B(z,2r) \cap \Omega$ and $B(z,r) \cap \Omega$ in the $\limsup$ above. Observe that if $Z$ is infinitesimally doubling, then every open set $\Omega \subset Z$ is infinitesimally doubling.  The key role of the infinitesimally doubling assumption in our main results is that it allows us to invoke the Lebesgue differentiation theorem (see \cite[Section 3.4]{HKST2015}).

A measure is \emph{$C_\mu$-doubling up to scale $R'$} if 
$$\mu( B(z,2r) ) \leq C_\mu \mu(B(z,r))$$ 
holds for every $z \in Z$ and $r \in (0,R')$. A measure is \emph{doubling} if we have $R' = \infty$. Clearly, being doubling up to a scale is much stronger than being infinitesimally doubling. For example, every measurable set of $\R^n$ equipped with the restricted $\mathcal L^n$-measure is infinitesimally doubling by the Lebesgue differentiation theorem while the doubling property can fail in domains, e.g. by considering a domain with an outward cusp, see \Cref{ex:2}.

We use the notation $\mathcal L^{\infty}(Z)$ for the space of essentially bounded functions $u\colon Z\to [-\infty,\infty]$. By identifying $u_1,u_2\in\mathcal L^{\infty}(Z)$ whenever $u_1=u_2$ almost everywhere, we obtain the quotient space $L^{\infty}(Z)$ with the norm given by $$\|u\|_{L^{\infty}(Z)}=\inf \{ \lambda\in\R\colon \mu\left( \{x\in Z\colon |u(x)|>\lambda\} \right) =0  \} .$$

\subsection{Lipschitz function spaces}
 For a metric measure space $Z$, a map $f\colon Z\to \mathbb R$ is said to be \emph{$L$-Lipschitz} for some constant $L\geq 0$ if \begin{equation}\label{eq:Lips_cond.}
|f(x)-f(y)| \leq Ld(x,y) \quad\text{for every $x,y\in Z$.}
\end{equation}
The smallest possible $L$ is denoted by $\lip(f)$. The collection of such functions is denoted by $\LIP( Z)$ and we endow the space with the seminorm $\|f\|=\lip (f)$.

\begin{definition}\label{def:lip}
The Banach space of bounded Lipschitz maps $\LIP^{\infty}(Z)$ is defined as the collection of all $f \in \LIP( Z )$ for which the following norm is finite:
$$\|f\|_{\LIP^{\infty}(Z)}= \|f\|_{L^{\infty}(Z)}+\lip(f).$$  
We refer to $\lip(f)$ as the (Lipschitz) energy seminorm.
\end{definition}

We claim that if $N \subset Z$ is negligible and $f \colon Z \setminus N \rightarrow \mathbb{R}$ is $L$-Lipschitz, then $f$ has a unique $L$-Lipschitz extension to $Z$. Indeed, since every open set has positive measure, the closure of $Z \setminus N$ coincides with $Z$. As $f$ maps Cauchy sequences to Cauchy sequences and $\mathbb{R}$ is complete, the existence of the $L$-Lipschitz extension $\widehat{f} \colon Z \rightarrow \mathbb{R}$ follows. In this manner, almost everywhere defined Lipschitz functions have unique extensions to $\LIP( Z)$.

\subsection[Hajlasz--Sobolev spaces]{Haj\l{}asz--Sobolev spaces}
Consider a metric measure space $Z$. Given a map $u \colon Z \rightarrow \mathbb{R}$, we say that a function $g \colon Z \rightarrow \left[0, \infty\right]$ is a \emph{Haj\l{}asz upper gradient} of $u$ if $g$ is measurable and there exists $N \subset Z$ with $\mu( N ) = 0$ such that
\begin{equation}\label{eq:Hineq}
    |u(x)-u(y)|
    \leq d(x,y)\lf(g(x)+g(y)\r) \quad \text{for every $x,y \in Z \setminus N$}.
\end{equation}
If, instead, \eqref{eq:Hineq} holds for every $x_0 \in Z$ and $x, y \in B( x_0, R ) \setminus N$ for a negligible set $N$ and $R > 0$ independent of $x_0$, we say that $g$ is a Haj\l{}asz gradient of $u$ \emph{up to scale $R$}.

The class of all Haj\l{}asz upper gradients of $u$ is denoted by $\mathcal{D}(u)$ and the Haj\l{}asz upper gradients of $u$ up to scale $R$ are denoted by $\mathcal{D}^{R}(u)$. We denote $\mathcal{D}_{\infty}(u) = \mathcal{D}(u) \cap \mathcal{L}^{\infty}(Z)$ and $\mathcal{D}_{\infty}^{R}(u) = \mathcal{D}^{R}(u) \cap \mathcal L^{\infty}(Z)$.  

Observe that if $u_1, u_{2}\colon Z\to\mathbb R$, $g \in \mathcal{D}^{R}( u_1 )$ and $u_2 = u_1$ almost everywhere, then $g \in \mathcal{D}^{R}( u_2 )$. In particular, if $u \in L^{\infty}( Z )$, the collections $\mathcal{D}( u )$ and $\mathcal{D}^R( u )$  can be unambiguously defined as those Haj\l{}asz gradients of \emph{some} (an arbitrary) representative of $u$. This motivates the following definition.
\begin{definition}\label{de:Mspace}
The Haj\l{}asz--Sobolev space $M^{1,\infty}(Z)$ is 
    \[  M^{1, \infty}(Z)\coloneqq\lf\{u\in L^{\infty}(Z) \colon \mathcal D_\infty(u)\neq\emptyset\r\},  \]
equipped with the norm 
\begin{equation}\label{eq:M-norm}
     \|u\|_{M^{1, \infty}(Z)}\coloneqq\|u\|_{L^\infty(Z)}+\inf_{g\in\mathcal D_\infty(u)}\|g\|_{L^\infty(Z)}. 
\end{equation}
For a parameter $R \in (0,\infty]$, we define the \emph{Haj\l{}asz energy up to scale $R$} by
\[\mathcal{E}_{\infty}^{R}( u ) \coloneqq \inf_{g\in\mathcal D_\infty^{R}(u)}\|g\|_{L^\infty(Z)}.\]
\end{definition}
One could renorm $M^{1,\infty}(Z)$ by using the energies $\mathcal{E}_{\infty}^{R}( u )$ instead of  $\mathcal{E}_{\infty}^{\infty}( u )$ on \eqref{eq:M-norm} and obtain an equivalent Banach space. This is a consequence of the following elementary lemma.
\begin{lemma}\label{lemm:local:to:global}
Let $u \in L^{\infty}(Z)$. If $R > 0$ and $g \in \mathcal{D}_{\infty}^{R}( u )$, then $\max\left\{ g, R^{-1} \|u\|_{ L^{\infty}(Z) } \right\} \in \mathcal{D}_{\infty}(u)$.
\end{lemma}
\begin{proof}
Let $N$ be as in the definition of $g \in \mathcal{D}_{\infty}^{R}(u)$ and denote $N_0 = N \cup \left\{ |u| > \|u\|_{ L^{\infty}(Z) } \right\}$. Consider $x, y \in Z \setminus N_0$. If $d( x, y ) < R$, then
\begin{equation*}
    | u(x)-u(y) | \leq (g(x)+g(y)) d(x,y).
\end{equation*}
Otherwise $d(x,y) \geq R$ and
\begin{equation*}
    | u(x) - u(y) | \leq 2\|u\|_{ L^{\infty}(Z) } \leq \frac{ 2 \|u\|_{ L^{\infty}(Z) } }{ R } d(x,y).
\end{equation*}
So the claim follows.
\end{proof}

The following result shows that the Haj\l{}asz--Sobolev space $M^{1, \infty}(Z)$ and the space of bounded Lipschitz maps are isomorphic.  We say that two Banach spaces $(\mathbb V,|\cdot|)$ and $(\mathbb W,\|\cdot\|)$ are $L$-isomorphic for some $L\geq 1$ if there is an isomorphism $T\colon\mathbb V\to\mathbb W$ so that $L^{-1}|u|\leq \|Tu\|\leq L|u|$
for all $u\in \mathbb V$. In the case $L=1$ we say that both spaces are isometric. 
\begin{lemma}\label{lemm:Lip:Haj}
Let $Z$ be a metric measure space. Then $\LIP^{\infty}( Z )$ and $M^{1,\infty}( Z ) $ are $2$-isomorphic as Banach spaces. More precisely, every $u\in \LIP^{\infty}(Z)$ belongs to an equivalence class in $M^{1,\infty}(Z)$ and every $u\in M^{1,\infty}(Z)$ has a representative $\widehat u\in \LIP^{\infty}(Z) $ that furthermore satisfies
\begin{equation*}
    2^{-1}\lip( \widehat{u} ) \in \mathcal{D}_{\infty}(u)
    \quad\text{and}\quad
    \lip( \widehat{u} )
    \leq
    2\|g\|_{ L^{\infty}(Z) }
    \quad\text{for $g \in \mathcal{D}_\infty(u)$.}
\end{equation*}
\end{lemma}
\begin{proof}
Observe that every $u\in \LIP^{\infty}( Z )$ satisfies $2^{-1}\lip(u) \in \mathcal{D}_{\infty}(  u )$. So  $u\in M^{1,\infty}( Z )$ with 
$$ \| u \|_{ M^{1,\infty}( Z ) }\leq\|  u \|_{ \LIP^{\infty}( Z ) }.$$ On the other hand, given $u\in M^{1,\infty}( Z ) $, fix a representative $\widehat u \in M^{1,\infty}(Z )$ and $g \in \mathcal{D}_{\infty}( \widehat u )$. Then there exists a negligible set $N \subset Z$ for which
\begin{equation*}
    | \widehat u(x) - \widehat u(y) |
    \leq
    2 \| g \|_{L^{\infty}(Z)} d(x,y)
    \quad\text{for every $x, y \in Z \setminus N$}.
\end{equation*}
Observe that every ball in $Z$ has positive measure so $Z \setminus N$ is dense in $Z$. Hence $\widehat u$ admits a unique $2\|g\|_{ L^{\infty}(Z) }$-Lipschitz extension to $Z$ that we still denote by $\widehat{u}$. The claim follows.
\end{proof}

\subsection{Newton--Sobolev spaces}\label{subsec_sob}
We continue on working with a metric measure space $Z$. A \emph{curve} is a continuous function $\gamma \colon \left[a, b\right] \rightarrow Z$ whose domain we  denote by $I_{\gamma} = \left[a,b\right]$ and its range by $|\gamma|=\{\gamma(t) \colon t\in [a,b]\}$. A curve $\gamma$ is \emph{rectifiable} if
\begin{equation}
    \label{eq:rectifiable}
    \ell( \gamma )
    =
    \sup¤
    \sum_{ i = 1 }^{ N } d( \gamma(t_i), \gamma( t_{i+1} ) )
    <
    \infty,
\end{equation}
where the supremum is taken over all finite partitions $t_1 = a,$ $t_{i} \leq t_{i+1}$, $t_{N+1} = b$. For a rectifiable curve, we denote
\begin{equation}
    \label{eq:metricspeed:pointwise}
    \| \gamma' \|(t)
    =
    \lim_{ s \rightarrow t }
        \frac{ d( \gamma(t), \gamma(s) ) }{ |t-s| },
\end{equation}
whenever the limit exists; the limit exists for $\mathcal{L}^{1}$-almost every $t \in \left[a,b\right]$ where $\mathcal{L}^{1}$ is the Lebesgue measure on $\mathbb{R}$. The mapping $t \mapsto \| \gamma' \|(t)$ is called the \emph{metric speed} of $\gamma$.

Every rectifiable curve $\gamma\colon [a, b]\to Z$ has a \emph{unit speed reparametrization} $\widehat{\gamma} \colon [0, l(\gamma)]\to Z$ with $\| \widehat{\gamma}'\|(t)  \equiv 1$ almost everywhere. The function $\widehat{\gamma}$ is the unique continuous map for which $\widehat{\gamma}( s(t) ) = \gamma( t )$ for $s(t) \coloneqq \ell( \gamma|_{ \left[a,t\right] } )$; see \cite[Eq. (5.1.6)]{HKST2015}. 
Given a Borel function $\rho \colon Z \rightarrow \left[0,\infty\right]$ and a rectifiable curve $\gamma$, we define
\begin{equation}
    \label{eq:curveintegral}
    \int_{\gamma} \rho \,ds
    \coloneqq
    \int_{ 0 }^{ \ell(\gamma) }
        \rho(\widehat{\gamma}(s))
    \,d\mathcal{L}^{1}(s).
\end{equation}
The term $\int_{\gamma} \rho \,ds$ is referred to as the \emph{(curve) integral} of $\rho$ over $\gamma$. curve integrals over constant curves are always zero. Recall that a function $\rho \colon Z \rightarrow \left[-\infty, \infty\right]$ is Borel if $\rho^{-1}( U )$ is a Borel subset of $Z$ for every open $U \subset [ -\infty, \infty ]$. The well-posedness of \eqref{eq:curveintegral} for Borel functions is proved in \cite[p. 126-127]{HKST2015}.

A rectifiable curve is \emph{not absolutely continuous} if there exists a negligible set $I \subset \left[a,b\right]$ and a Borel set $B \subset \gamma( I )$ with $\int_{\gamma} \chi_{B} \,ds > 0$. If no such $I$ and $B$ exist, we say that $\gamma$ is \emph{absolutely continuous}. For an absolutely continuous curve $\gamma \colon \left[a,b\right] \rightarrow Z$, we have
\begin{equation*}
    \label{eq:curveintegral:ABS}
    \int_{\gamma} \rho \,ds
    =
    \int_{ a }^{ b }
        ( \rho( \gamma(s) ) ) \| \gamma' \|(s)
    \,d\mathcal{L}^{1}(s),
    \quad\text{see \cite[Section 5.1]{HKST2015}}.
\end{equation*}
For every rectifiable curve $\gamma\colon[a,b]\to Z$, its unit speed reparametrization $\widehat \gamma\colon[0,\ell(\gamma)]\to  Z$ is absolutely continuous and $\|\widehat \gamma '\|(t)=1$ at almost every $t$ (see \cite[Proposition 5.1.8]{HKST2015}).

Any collection of curves $\gamma \colon I_{\gamma} \rightarrow Z$ is referred to as a \emph{curve family}, typically denoted by $\Gamma$.
\begin{definition}\label{de:negligible}
A curve family $\Gamma$ in $Z$ is \emph{$\infty$-negligible} if there exists a nonnegative Borel $\rho \in L^{\infty}( Z )$ such that $\| \rho \|_{ L^{\infty}(Z) }= 0$ and $\int_{\gamma} \rho \,ds = \infty$ for every rectifiable $\gamma \in \Gamma$.

A property holds for \emph{$\infty$-almost every} $\gamma$ if the family of curves where the property fails is $\infty$-negligible. For a Borel set $N \subset Z$, the curve family $\Gamma_{N}^{+}$ consists of the rectifiable curves $\gamma$ for which $\int_{\gamma} \chi_{N}\,ds > 0$.
\end{definition}
By definition, the collection of nonrectifiable curves is $\infty$-negligible. Moreover, any curve family containing a constant curve is \emph{not} $\infty$-negligible. Note as well that our notion of \emph{$\infty$-negligibility} coincides with the usual one defined using the \emph{$\infty$-modulus}; see \cite[Lemma 5.7]{Durand}. In particular, it follows from \Cref{de:negligible} that for a null Borel set $N\subset Z$ the family $\Gamma^+_N$ is $\infty$-negligible. In fact, it follows from \Cref{de:negligible} that a curve family $\Gamma$ is $\infty$-negligible if and only if $\Gamma \subset \Gamma^{+}_{N}$ for some negligible Borel set $N \subset Z$.

\begin{definition}
Let $u \colon Z \rightarrow \mathbb{R}$. A mapping $\rho \colon Z \rightarrow [0,\infty]$ is a \emph{$\infty$-weak upper gradient} of $u$ if $\rho$ is Borel measurable and for $\infty$-almost every nonconstant $\gamma \colon \left[a,b\right] \rightarrow Z$,
\begin{equation}\label{eq:boundaryinequality}
    |u( \gamma(a) ) - u( \gamma(b) )|
    \leq
    \int_{\gamma} \rho \,ds.
\end{equation}
If, in addition, $\rho \in \mathcal L^\infty(Z)$, we denote $\rho \in \mathcal{D}_{N,\infty}(u)$.
\end{definition}

\begin{lemma}\label{lemm:gradient:absolutecontinuity}
Let $u \colon Z \rightarrow \mathbb{R}$ and $\rho \in \mathcal{D}_{N,\infty}( u )$. There exists a $\infty$-negligible family $\Gamma_{u,\rho}$ containing $\left\{ \gamma \colon \int_{\gamma} \rho \, ds = \infty \right\}$ such that $u \circ \gamma$ is absolutely continuous and
\begin{equation}\label{eq:upper:abs}
    \| ( u \circ \gamma )' \|(t)
    \leq 
    \rho( \gamma(t) ) \| \gamma' \|(t)
    \quad\text{for almost every $t$}
\end{equation}
for every nonconstant absolutely continuous $\gamma \not\in \Gamma_{u,\rho}$
\end{lemma}
\begin{proof}
The claim is standard. We refer the reader e.g. to \cite[Lemma 3.5]{Garc:Iko:Zhu:23} for a proof.
\end{proof}

\begin{lemma}\label{lem:minimal-upper-grad.}
    For every measurable $u\colon Z\to \R$ for which $\mathcal D_{N,\infty}(u)\neq \emptyset$ there exists $\rho\in \mathcal D_{N,\infty}(u)$ such that for any other $\widetilde \rho \in \mathcal D_{N,\infty}(u)$ we have $\rho(x)\leq \widetilde \rho(x)$ almost everywhere. 
    \end{lemma}  
\begin{proof}
    See for instance \cite{M2013}.
\end{proof}
The $L^p$-equivalence class of $\rho$, from \Cref{lem:minimal-upper-grad.} is denoted by $\rho_u$ and is referred to as the {\em minimal $\infty$-weak upper gradient of $u$}. We next define the Newton--Sobolev space $N^{1,\infty}(Z)$ as follows. 
\begin{definition}\label{def:N1P}
We say $\widehat{u} \in \widehat{N}^{1,\infty}( Z )$ if $\widehat{u} \in \mathcal{L}^\infty(Z)$ and $\mathcal{D}_{N,\infty}( \widehat{u} ) \neq \emptyset$. We denote
\begin{equation*}
    \| \widehat{u} \|_{ \widehat{N}^{1,\infty}(Z) }
    \coloneqq
    \| \widehat{u} \|_{L^\infty(Z)}
    +
    \inf_{ \rho \in \mathcal{D}_{N,\infty}( \widehat{u} ) }\| \rho\|_{L^\infty(Z)}.
\end{equation*}
Two elements $\widehat{u}_{1}, \widehat{u}_{2} \in \widehat{N}^{1,\infty}( Z )$ are identified if $\| \widehat{u}_{1} - \widehat{u}_{2} \|_{ \widehat{N}^{1,\infty}(Z) } = 0$. With this convention, we obtain from $\widehat{N}^{1,\infty}( Z )$  a Banach space $N^{1,\infty}( Z; \mathbb{V} )$ with the norm
\begin{equation*}
    \| u \|_{ N^{1,\infty}(Z ) }
    \coloneqq
    \| \widehat{u} \|_{ \widehat{N}^{1,\infty}(Z) },
\end{equation*}
where $\widehat{u}$ is any $\widehat{N}^{1,\infty}( Z )$-representative of ${u} \in N^{1,\infty}(Z) $; see \cite[Section 7.1]{HKST2015}. We refer to $\mathscr{E}_{\infty}(u) \coloneqq \inf_{ \rho \in \mathcal{D}_{N,\infty}( \widehat{u} ) }\| \rho\|_{L^\infty(Z)}=\|\rho_u\|_{L^\infty(Z)}$ as the (Newtonian) energy seminorm.
\end{definition}
In this paper we work with the following variant of the $N^{1,\infty}$-space.
\begin{definition}\label{def:Dir}
We denote $u \in W^{1,\infty}( Z )$ if $u \in L^{\infty}( Z )$ and there exists $\widehat{u} \in \widehat{N}^{1,\infty}( Z )$ for which $\left\{ u \neq \widehat{u} \right\}$ is negligible. The space $W^{1,\infty}( Z )$ is equipped with the norm
\begin{equation*}
    \|u\|_{W^{1, \infty}(Z)}
    \coloneqq
    \| \widehat{u} \|_{ \widehat{N}^{1,\infty}(Z) }.
\end{equation*}
\end{definition}
Lemma 5.13 in \cite{Durand} implies that if $\widehat{u}_1, \widehat{u}_2 \in \widehat{N}^{1,\infty}(Z)$ agree almost everywhere, then
\begin{equation*}
    \inf_{ \rho \in \mathcal{D}_{N,\infty}( \widehat{u}_2 - \widehat{u}_1 ) }\| \rho\|_{L^\infty(Z)}
    =
    0
    \quad\text{and}\quad
    \inf_{ \rho \in \mathcal{D}_{N,\infty}( \widehat{u}_2 ) }\| \rho\|_{L^\infty(Z)}
    =
    \inf_{ \rho \in \mathcal{D}_{N,\infty}( \widehat{u}_1 ) }\| \rho\|_{L^\infty(Z)}
\end{equation*}
The following lemma gives some embeddings between the spaces introduced in this section.
\begin{lemma}\label{lemma:M:in:W}
Let $Z$ be a metric measure space. 
\begin{enumerate}
    \item The linear mapping $P\colon\LIP^{\infty}(Z)\to N^{1,\infty}(Z)$  defined by $u\mapsto u$ is injective and $1$-Lipschitz
    \item The linear mapping $\iota \colon M^{1,\infty}( Z ) \rightarrow W^{1,\infty}( Z )$ defined by $u \mapsto u$ is injective and $2$-Lipschitz. In particular, every $u \in M^{1,\infty}( Z )$ induces a unique equivalence class $\widehat{u} \in N^{1,\infty}( Z )$.
\end{enumerate}\end{lemma}

\begin{proof}
For $(1)$ it is enough to observe that $L$-Lipschitz functions admit $\rho=L$ as a $\infty$-weak upper gradient. For $(2)$, consider a function $u\in M^{1,\infty}(Z)$. By \Cref{lemm:Lip:Haj}, there exists a Lipschitz representative $\widehat u$ of $u$ with $\|\widehat u\|_{\text{LIP}^{\infty}(Z)}\leq 2 \|u\|_{M^{1,\infty}(Z)}$. Then $(1)$  yields that $\widehat u\in W^{1,\infty}(Z)$ with $\|u\|_{W^{1,\infty}(Z)}=\|\widehat u\|_{\widehat{N}^{1,\infty}(Z)}\leq |\widehat u\|_{\text{LIP}^{\infty}(Z)}\leq 2 \|u\|_{M^{1,\infty}(Z)}$
\end{proof}
From now on we mainly use the spaces $M^{1,\infty}$ and $W^{1,\infty}$ instead of $\LIP^{\infty}$ and $N^{1,\infty}$, respectively. The results can be translated back and forth since, by \Cref{lemm:Lip:Haj}, the spaces $M^{1,\infty}(Z)$ and $\LIP^{\infty}(Z)$ are $2$-isomorphic and the spaces $W^{1,\infty}(Z)$ and $N^{1,\infty}(Z)$ are isometrically isomorphic.

\section{Thick quasiconvexity, Poincaré inequality, and Lipschitz representatives}\label{sec:thick}
The main goal of this section is to give the proofs of \Cref{thm:equivalence,thm:bounded:doubling}. However, we also introduce thick quasiconvexity and prove basic properties. We also recall the definition of a weak Poincaré inequality.

\subsection{Thick quasiconvexity}

\begin{definition}\label{def:quasi_R}
A  metric measure space $Z$ is $(C,R)$-quasiconvex (or uniformly locally quasiconvex) if there exist $C\geq 1$ and $R\in (0,\infty]$ such that, for every pair of points $x,y\in Z$ with $d( x, y )<R$, there exists a rectifiable curve $\gamma$ connecting $x$ to $y$ so that $\ell(\gamma)\leq Cd(x,y) $. If $R=\infty$ we say that $Z$ is ($C$)-quasiconvex. 

\end{definition}

Quasiconvexity is a purely geometric notion. A closely related notion, where the measure plays a role, appeared in \cite[Definition 4.1]{Durand_2}; see also \cite{DeCe:Pal:88}. This is the notion of $\infty$-thick quasiconvexity. More recently, an equivalent definition of $\infty$-thick quasiconvexity was given in \cite{Creutz_Soultanis} for infinitesimally doubling spaces. We localize and adapt both definitions from \cite{Durand_2,Creutz_Soultanis} as follows.

\begin{definition}\label{def:quasi_Z}
Let $C \geq 1$ and $R \in (0,\infty]$. A metric measure space $Z$ is \emph{$\infty$-thick $(C,R)$-quasiconvex} if for every $x_0\in Z$ and all measurable sets $E, F \subset  B( x_0, R )$ with $\mu(E), \mu(F) > 0$, the family of curves 
$$\Gamma( E, F; C )=\{\gamma \colon \left[0,1\right] \rightarrow Z \colon \gamma(0) = x \in E,\; \gamma(1) = y \in F\;\;\text{and}\;\;\ell( \gamma ) \leq C d( x, y )\}$$ is not $\infty$-negligible. If $R=\infty$, we say that $Z$ is $\infty$-thick ($C$)-quasiconvex.
\end{definition}
We also need the following stronger version.
\begin{definition}\label{def:verythick}
Let $C\geq 1$ and $R\in(0, \infty]$. A metric measure space $Z$ is very $\infty$-thick $(C, R)$-quasiconvex if for every $x_0\in Z$ and every $x, y\in B(x_0, R)$, the family of curves 
\[\Gamma(x, y; C)=\left\{\gamma\colon[0, 1]\to Z\colon\gamma(0)=x, \gamma(1)=y\ {\rm and}\ \ell(\gamma)\leq Cd(x, y))\right\}\]
is not $\infty$-negligible. If $R=\infty$, we say that $Z$ is very $\infty$-thick ($C$)-quasiconvex.
\end{definition}
The $\infty$-thick quasiconvexity can be strengthened to very $\infty$-thick quasiconvexity under mild assumptions by the use of the following lemma. The following lemma is a local version of \cite[Theorem 3.1]{D-CJS2016}, where $Z$ was assumed to be  complete and doubling. Since the proof is essentially the same, we omit it.

\begin{lemma}\label{thm:very_thick-iff-infty-thick}
Let $Z$ be a locally complete metric measure space which supports an infinitesimally doubling Borel measure $\mu$ which is nontrivial and finite on every ball. Then the following are quantitatively equivalent:
\begin{enumerate}
    \item  $Z$ is $\infty$-thick $(C, R)$-quasiconvex;
    \item  $Z$ is very $\infty$-thick $(C', R)$-quasiconvex;
    \item  There exist $C''\geq 1$ and $R\in(0, \infty]$ such that for every null set $N\subset Z$, and for every $x_0\in Z$ and every $x, y\in B(x_0, R)$, there exists a $C''$-quasiconvex curve $\gamma$ in $Z$ connecting $x$ to $y$ with $\gamma\notin\Gamma_N^+$.
\end{enumerate}
In fact, $(2) \Rightarrow (1)$ holds with the same constants $C = C'$ and $R$, $(1) \Rightarrow (3)$ holds with constants $C'' = 6C$ and $R$, and $(3) \Rightarrow (2)$ holds with $C' = C''$ and the same $R$.
\end{lemma}

The lemma leads to the following immediate self-improvement.
\begin{corollary}\label{cor:thick-to-very-thick}
Let $Z$ be a locally complete metric measure space which supports an infinitesimally doubling Borel measure $\mu$ which is nontrivial and finite on every ball. If $Z$ is $\infty$-thick $(C,R)$-quasiconvex, then $Z$ is very $\infty$-thick $(C',R)$-quasiconvex for every $C' > C$.  
\end{corollary}
\begin{proof}
Let $C' > C$. Consider a distinct pair $x, y \in B( x_0, R )$ and $\epsilon > 0$ such that $\epsilon < ( C'-C)d(x,y)/(14C)$.

Let $\Gamma_0$ be an $\infty$-negligible curve family. Recall that it holds that $\Gamma_0 \subset \Gamma^{+}_{B}$ for some negligible Borel set $B$. By the $\infty$-thick $(C,R)$-quasiconvexity, there exists a $C$-quasiconvex curve $\sigma \not\in \Gamma^{+}_{B}$ joining a point $x'\in B( x, \epsilon ) \cap B( x_0, R )$ to $y'\in B( y, \epsilon ) \cap B( x_0, R )$. By the very $\infty$-thick $( 6C, R )$-quasiconvexity, there exist $6C$-quasiconvex curves $\sigma_x, \sigma_y \not\in \Gamma^{+}_{B}$ joining $x$ to $x'$ and $y'$ to $x$, respectively. The concatenation $\gamma = \sigma_x \star \sigma \star \sigma_y$ satisfies $\ell( \gamma ) \leq C' d(x,y)$ and $\gamma \not\in \Gamma^{+}_{B}$. The claim follows.
\end{proof}
The punctured disc in the plane illustrates that \Cref{cor:thick-to-very-thick} is sharp. The point is that the punctured disc $\{(x,y)\in\R^2:\, 0<x^2+y^2<1\}$ is $\infty$-thick $1$-quasiconvex but it is not very $\infty$-thick $1$-quasiconvex.

\begin{remark}\label{rem:verythick-to-thick}
Very $\infty$-thick $( C, R )$-quasiconvexity is stronger than $( C, R )$-quasiconvexity. On the other hand, quasiconvexity does not necessarily imply $\infty$-thick $( C, R )$-quasiconvexity for any $C$ and $R$, even for complete doubling spaces. The Sierpi\'nski carpet $S$ is such an example as discussed in the introduction. 
\end{remark}

The very $\infty$-thick quasiconvexity is inherited by uniformly locally quasiconvex open subsets in the following sense.
\begin{lemma}\label{lem:thick-to-verythick}
Suppose that $Z$ is very $\infty$-thick $(C,R)$-quasiconvex, and $\Omega \subset Z$ is open and $( \widetilde{C}, R' )$-quasiconvex. Then $\Omega$ is very $\infty$-thick $( \widetilde{C}C, R' )$-quasiconvex.
\end{lemma}
\begin{proof}
Let $x_0 \in \Omega$ and consider a distinct pair $x, y \in \Omega \cap B( x_0, R' )$. Let $\Gamma$ be an $\infty$-negligible curve family. There exists a $\mu$-negligible Borel set $B \subset Z$ such that $\Gamma \subset \Gamma^{+}_{B}$.

There exists a $\widetilde{C}$-quasiconvex curve $\sigma \colon [0,1] \rightarrow \Omega$ joining $x$ to $y$. Let $\delta = C \inf d( \sigma(t), z ) > 0$ where the infimum is taken over $t \in [0,1]$ and $z \in Z \setminus \Omega$. Consider a partition $\{ t_i \}_{ i = 0 }^{ n+1 }$ of $[0,1]$ such that $\ell( \sigma|_{ [t_i, t_{i+1}] } ) < \min\{ \delta, R \}$ for $0 \leq i \leq n$. There exists $\gamma_i \colon [0,1] \rightarrow Z \not\in \Gamma^{+}_{B}$ that is $C$-quasiconvex and joins $\sigma( t_i )$ to $\sigma( t_{i+1} )$ for $0 \leq i \leq n$. The definition of $\delta$ yields that $\gamma_i$ takes values in $\Omega$ for $0 \leq i \leq n$. The concatenation $\gamma = \gamma_0 \star \dots \star \gamma_{n} \colon [0,1] \rightarrow \Omega \not\in \Gamma^{+}_{B}$ satisfies
\begin{align*}
    \ell( \gamma )
    =
    \sum_{ i = 1 }^{ n } \ell( \gamma_i )
    \leq
    \sum_{ i = 1 }^{ n } C d( \sigma( t_i ), \sigma( t_{i+1} ) )
    \leq
    C \ell( \sigma )
    \leq
    C \widetilde{C} d( x, y ).
\end{align*}
Since $\Gamma$ was an arbitrary $\infty$-negligible curve family, the claim follows.
\end{proof}

Using the self-improvement in \Cref{cor:thick-to-very-thick}, the completion also inherits very $\infty$-thick quasiconvexity.
\begin{lemma}\label{lemm:equivalence:completion}
Let $Z$ be a locally complete metric measure space with a reference measure that is infinitesimally doubling, and finite and positive on all balls. Let $\overline{Z}$ be a metric completion of $Z$. If $Z$ is very $\infty$-thick $(C,R)$-quasiconvex, then $\overline{Z}$ is very $\infty$-thick $(C',R)$-quasiconvex for every $C' > C$.
\end{lemma}
\begin{proof}
Suppose that $Z$ is very $\infty$-thick $(C,R)$-quasiconvex. It readily follows that the metric completion $\overline{Z}$ is $\infty$-thick $(C,R)$-quasiconvex. By \Cref{cor:thick-to-very-thick}, $\overline{Z}$ is very $\infty$-thick $(C',R)$-quasiconvex for every $C' > C$.
\end{proof}

\subsection{Poincaré inequalities}
For a locally integrable function $u\in\mathcal L^{1}_{loc}(Z)$ and a ball $B = B( x_0, r )$, the integral average is denoted by
 $$\aint{B}u \, d\mu=\dfrac{1}{\mu(B)}\int_{B} u\, d\mu. $$
We also denote $\lambda B \coloneqq B( x_0, \lambda r )$ for $\lambda > 0$.
 
\begin{definition}\label{def:PI-ineq}
A metric measure space $Z$ supports a weak $(1,p)$-Poincaré inequality, for some $1\leq p<\infty$, if there exist constants $C>0,\lambda\geq 1$ such that
\begin{equation}\label{eq:Poin_ineq}
    \aint{B} |u-u_B|\,d\mu\leq C\diam (B)\left(\aint{\lambda B} \rho^p\,d\mu  \right)^{1/p}
\end{equation}
for every ball $B\subset Z$, every function $u\colon Z\to \R$ integrable in bounded sets, and every $p$-weak upper gradient $\rho$ of $u$. In case $p=\infty$, we instead require
\begin{equation}\label{eq:Poin_ineq_infinity}
    \aint{B} |u-u_B|\,d\mu\leq C\diam (B)\| \rho\|_{L^{\infty}(\lambda B)},
\end{equation}
where $\rho$ is any $\infty$-weak upper gradient of $u$. If \eqref{eq:Poin_ineq} (resp. \eqref{eq:Poin_ineq_infinity}) hold for all balls $B$ whose radius is at most $r$, we say that $Z$ supports a weak $(1,p)$-Poincaré (resp. $(1,\infty)$-Poincaré) inequality up to scale $r$ with constants $(C,\lambda)$.
\end{definition}

Observe that a metric measure space $Z$ that supports a Poincaré inequality is necessarily connected. However, disconnected spaces can have Poincaré inequalities up to a scale. Consider, for instance, the union of two separated balls in the Euclidean space.

It is easy to see using Hölder's inequality that if $Z$ supports a weak $(1,q)$-Poincaré inequality for some $1\leq q<\infty$ up to scale $r$ with constants $(C,\lambda)$, then it supports a weak $(1,p)$-Poincaré inequality for any $p\in [q,\infty]$ up to scale $r$ with constants $(C,\lambda)$. A seminal result by Keith and Zhong \cite{KZ2008} proves that when the measure of $Z$ is doubling and $Z$ is complete, then the weak Poincaré inequality is an open ended condition. More recently, their result was localized by Björn and Björn \cite{Bj:Bj:18}. For the case $p=\infty$, there are examples of complete metric measure spaces with a doubling measure which have the weak $(1,\infty)$-Poincaré inequality but still do not satisfy any weak $(1,p)$-Poincaré inequality for any $1\leq p<\infty$ (see \cite[Example 2]{DSW}).

Several authors have explored the relations between Poincaré inequalities and quasiconvexity. For instance, it is known that any locally complete metric measure space $Z$ that has a doubling measure and supports a weak $(1,p)$-Poincaré inequality for some $p\in [1,\infty]$ is necessarily $\infty$-thick quasiconvex (see e.g. \cite[Proposition 4.3]{Durand_2}), therefore quasiconvex too. We include a uniform local version of this result in \Cref{prop:quasiconvexity_2}. Furthermore, from $\infty$-thick $(C,R)$-quasiconvexity one obtains a weak $(1,\infty)$-Poincaré inequality up to some scale $R'>0$ for metric measure spaces that are locally complete and infinitesimally doubling. This follows from \Cref{thm:equivalence}.

 \subsection{Proofs of \Cref{thm:equivalence,thm:bounded:doubling}}\label{sec:p=infty}

We divide the proof of \Cref{thm:equivalence} into several propositions; \Cref{prop:thick:to:Hajlasz,prop:Hajlasz:to:thick,prop:quasiconvexity_2}. The equivalence between the equality $\LIP^{\infty}(Z)=W^{1,\infty}(Z)$ and the $\infty$-thick $(C,R)$-quasiconvexity of $Z$ will be a consequence of \Cref{prop:thick:to:Hajlasz,prop:Hajlasz:to:thick}.

\begin{proposition}\label{prop:thick:to:Hajlasz}
Let $Z$ be a metric measure space that is infinitesimally doubling. If $Z$ is $\infty$-thick $(C,R)$-quasiconvex and $\widehat{u} \colon Z \to \mathbb{R}$ is a measurable function with an $\infty$-weak upper gradient $\rho$ and $C' > C$, then $z \mapsto C\| \rho \|_{ L^{\infty}( B( z, C'R) ) }$ is a Haj{\l}asz gradient of $\widehat{u}$ up to scale $R$. In particular, $C\|\rho\|_{ L^{\infty}( Z ) }$ is a Haj{\l}asz gradient of $u$ up to scale $R$. Moreover, $M^{1,\infty}(Z)=W^{1,\infty}(Z)$ with quantitatively equivalent norms.
\end{proposition}
\begin{proof}

Consider a measurable function $\widehat{u} \colon Z \to \mathbb{R}$ with an $\infty$-weak upper gradient $\rho$. By \Cref{lemm:gradient:absolutecontinuity}, there exists an $\infty$-negligible family $\Gamma_{ \widehat{u}, \rho }$ such that for every absolutely continuous $\gamma \not\in \Gamma_{ \widehat{u}, \rho }$, the composition $\widehat u \circ \gamma$ is absolutely continuous and satisfies
\begin{equation*}
    \| ( \widehat{u} \circ \gamma )' \|(t)
    \leq 
    \rho( \gamma(t) ) \| \gamma' \|(t)
    \quad\text{for almost every $t$.}
\end{equation*}
Consider the $\mu$-negligible set $E = \left\{z\in Z \colon\rho(z) > \| \rho \|_{L^\infty(Z)} \right\}$ and the $\infty$-negligible family 
$$\Gamma_0 = \left\{ \gamma \colon \int_{\gamma} \chi_{E} \,ds > 0 \right\} \cup \Gamma_{ \widehat{u}, \rho }.$$
For every nonconstant absolutely continuous $\gamma \not\in \Gamma_{0}$, the composition $\widehat u\circ \gamma$ is absolutely continuous, and
\begin{equation}\label{eq:p-thick-quasiconvexity}
    | ( \widehat{u} \circ \gamma )' |(t)
    \leq
    \|\rho\|_{ L^\infty(U) } \| \gamma' \|(t)
    \quad\text{for almost every $t$ for any open set $U \supset |\gamma|$.}
\end{equation}
Fix $x_{0} \in Z$ and $U = B( x_0, R )$. By the Lusin–Egoroff theorem, for each $\eta > 0$, there exists a Borel set $K_{\eta} \subset U$ such that $\mu( U \setminus K_\eta ) < \eta$ and $\widehat u|_{ K_\eta }$ is continuous. Since $\mu( U ) > 0$, we may assume that $\mu( K_{\eta} ) > 0$ by taking a smaller $\eta$ if needed.

Let $G_{ \eta } \subset K_{\eta}$ denote the set of Lebesgue density points of $K_\eta$. In particular, due to the infinitesimally doubling condition, the Lebesgue differentiation theorem implies that $\mu (K_{\eta}\setminus G_{\eta})=0$. For every pair $x,y \in G_{\eta}$, consider $r_0 > 0$ such that for all $0 < r< r_0$, we have $B( x, r )\cup B(y,r) \subset B( x_0, R )$ and $\overline{B}( x, r ) \cap \overline{B}( y, r ) = \emptyset$. Observe that $\mu( B( x, r ) \cap G_\eta ) > 0$ and $\mu( B( y, r ) \cap G_\eta ) > 0$ by definition of $G_\eta$.

By the $\infty$-thick $(C,R)$-quasiconvexity, the family $\Gamma( B( x, r ) \cap G_{\eta}, B( y, r ) \cap G_{\eta}; C )$ is not $\infty$-negligible for $r \in (0,r_0)$. In particular, for every $r\in (0,r_0)$, there exists a nonconstant absolutely continuous curve $\gamma_r \colon [0,1] \rightarrow Z$ in $\Gamma( B( x, r ) \cap G_{\eta}, B(y, r ) \cap G_{\eta}; C )\setminus\Gamma_0$. Denote $x_r = \gamma_r(0)$ and $y_r = \gamma_r(1)$, and fix $s > R$. By the $C$-quasiconvexity of $\gamma_r$, it holds that
$$|\gamma_r| \subset \overline{B}( x_r, 2^{-1}Cd(x_r,y_r) ) \cup \overline{B}( y_r, 2^{-1}Cd(x_r,y_r)) \subset B( x_r, C s) \cup B( y_r, Cs ).$$
Motivated by this inclusion and \eqref{eq:p-thick-quasiconvexity}, we denote
\begin{align*}
    \widehat{g}_{R}(z) \coloneqq C\lim_{ s \rightarrow R^{+} }\| \rho \|_{ L^{\infty}( B( z, Cs) ) }.
\end{align*}
Notice that $\widehat{g}_{R}$ is upper semicontinuous. Integrating \eqref{eq:p-thick-quasiconvexity} implies that
\begin{align}\label{eq:p-thick-quasiconvexity:mod}
    | \widehat{u}( \gamma_r(1) ) - \widehat{u}( \gamma_r(0) ) |
    &\leq 
    \lim_{ s \rightarrow R^{+} }
    \|\rho\|_{L^{\infty}(B(x_r,Cs)\cup B(y_r,Cs))}\ell(\gamma_r) \nonumber\\
    &\leq 
    ( \widehat{g}_{R}( \gamma_r(0) ) + \widehat{g}_{R}( \gamma_r(1)) )
    d( \gamma_r(1), \gamma_r(0) ).
\end{align}
By passing to the limit $r \rightarrow 0^{+}$, the continuity of $\widehat{u}|_{ G_\eta }$, inequality \eqref{eq:p-thick-quasiconvexity:mod}, and the upper semicontinuity of $\widehat{g}_{R}$ imply 
\begin{equation}\label{eq:Lipsc-continuity}
    | \widehat{u}( x) - \widehat{u}(y ) |
    \leq
    ( \widehat{g}_{R}( x ) + \widehat{g}_{R}( y ) )d( x, y )
    \quad\text{for every $x, y \in G_{\eta} \cap B( x_0, R )$.}
\end{equation}
By passing to the limit $\eta \rightarrow 0^{+}$, we deduce that \eqref{eq:Lipsc-continuity} holds for every pair $x, y \in B( x_0, R ) \setminus N$ for some negligible set $\mu( N ) = 0$. We cover $Z$ by balls $\left\{ B( x_{0}, R ) \right\}_{ x_0 \in Z }$ and apply the Lindelöf property of $Z$ to obtain a countable subcover $\left\{ B( x_i, R ) \right\}_{ i \in \mathbb{N} }$ of $Z$. Since \eqref{eq:Lipsc-continuity} holds for every $x,y \in B( x_i, R ) \setminus N_i$ outside some negligible set $N_i \subset B( x_i, R )$, $i \in \mathbb{N}$, we conclude that \eqref{eq:Lipsc-continuity} holds for every $x, y \in B( z, R ) \setminus \bigcup_{ i \in \mathbb{N} } N_i$ for every $z \in Z$. This means that  $\widehat{g}_{R}(z)$ is a Haj{\l}asz gradient of $\widehat u$ up to scale $R>0$. In particular, for every $C'>C$, it holds that $C\|\rho\|_{L^{\infty}(B(z, C'R))}$ is a Haj{\l}asz gradient of $\widehat u$ up to scale $R>0$ and
\begin{equation*}
    \inf\left\{ \| g \|_{ L^{\infty}(Z) } \mid \text{$g$ is a Haj\l{}asz gradient up to scale $R$} \right\}
    \leq
    C \| \rho \|_{ L^{\infty}(Z) }.
\end{equation*}
This gives the first part of the claim.

Recall that there is a $2$-Lipschitz linear embedding $M^{1,\infty}(Z) \xhookrightarrow{} W^{1,\infty}(Z)$ by \Cref{lemma:M:in:W}. Then \Cref{lemm:local:to:global} and the argument above give that $W^{1,\infty}(Z) = M^{1,\infty}(Z)$ with quantitatively equivalent norms.
\end{proof}

\begin{proposition}\label{prop:Hajlasz:to:thick}
Let $Z$ be a metric measure space that is infinitesimally doubling and locally complete. If there are $C > 1$ and $R > 0$ such that for every $\widehat{u} \in \widehat{N}^{1,\infty}( Z )$ and $\rho \in \mathcal{D}_{N,\infty}( \widehat{u} )$, the constant function $2^{-1}C \| \rho \|_{ L^{\infty}(Z) }$ is a Haj{\l}asz gradient of $u$ up to scale $R$, then $Z$ is $\infty$-thick $(C',R)$-quasiconvex for every $C' > C$.
\end{proposition}

\begin{proof}
We suppose that the claim is not true and derive a contradiction; fix $C' > C$. Assume that there exist $z \in Z$ and measurable sets $E, F \subset B( z, R )$ with $\mu( E ), \mu( F ) > 0$ so that $\Gamma( E, F; C' )$ is $\infty$-negligible. We consider a distinct pair of Lebesgue density points $( x', y' ) \in E \times F$ and the sets $E_r = E \cap B( x', r )$ and $F_r = F \cap B(y',r)$ for small $r > 0$.  Note that the existence of such pair $(x',y')$ is guaranteed by the Lebesgue differentiation theorem, which in turn follows from the infinitesimally doubling assumption. Then
\begin{equation*}
    a(r) = \max\left\{ \diam E_{r}, \diam F_{r} \right\} \leq 2r
    \quad\text{and}\quad
    b(r) = d( E_{r}, F_{r} ) 
\end{equation*}
are decreasing and increasing, respectively, and $\lim_{ r \rightarrow 0^{+} } a(r) = 0$ and $\lim_{ r \rightarrow 0^{+} } b(r) = d(x',y') > 0$. We consider $r > 0$ so small that $b(r) > 0$ and define $\lambda_r = a(r)/b(r)$. Observe now that $\Gamma( E_{r}, F_{r}; C' ) \subset \Gamma( E, F; C' )$ is still $\infty$-negligible. Since $\lambda_r$ converges to zero as $r \rightarrow 0^{+}$, we may fix $r > 0$ so small that
\begin{equation}\label{eq:smallenough}
    \lambda_r < 1 \quad\text{and}\quad \frac{ C' }{ 1 + 2 \lambda_r } > C.
\end{equation}
We use \eqref{eq:smallenough} to derive a contradiction. We first observe that
\begin{equation*}
    b(r) \geq d(x,y) - 2a(r) = d(x,y) - 2 \lambda_r b(r)
    \quad\text{when $(x,y) \in E_r \times F_r$}
\end{equation*}
and so
\begin{align}\label{eq:pointwiseinequality}
    b(r) \geq \frac{ d(x,y) }{ 1 + 2 \lambda_r }
    \quad\text{for every $(x,y) \in E_r \times F_r$.}
\end{align}
Since $\Gamma( E_r, F_r; C' )$ is $\infty$-negligible, there exists a  nonnegative Borel function $\rho\in \mathcal{L}^{\infty}( Z )$ so that  $\int_{\gamma} \rho\, ds=\infty$ for every rectifiable $\gamma\in \Gamma( E_r, F_r; C' )$. On the other hand, since $$\mu (\{z \colon \rho(z)>\|\rho\|_{L^{\infty}(Z)}\})=0\quad \Rightarrow\quad \left\{\gamma \colon \int_\gamma \chi_{\left\{\rho > \|\rho\|_{L^{\infty}(Z)}  \right\}}ds>0\right\} \;\text{is $\infty$-negligible,}$$ we conclude that 
\begin{equation}\label{eq:negligibleset}
    \Gamma_{0}
    \coloneqq
    \left\{
        \gamma
        \colon
        \int_{ \gamma } \rho + \infty \cdot \chi_{ \left\{ \rho > \|\rho\|_{L^{\infty}(Z)} \right\} } \,ds = \infty
    \right\}
\end{equation}
is $\infty$-negligible and $\Gamma( E_r, F_r; C' ) \subset \Gamma_0$.

Define a function $v \colon Z \rightarrow \left[0, \infty\right]$ as follows:
\begin{equation*}
    v(x)
    =
    \lim_{ k \rightarrow \infty }\left(
    \inf\left\{
        \int_{ \gamma } \left( \chi_Z + \frac{\rho}{k} \right) \,ds
        \colon
        \gamma(0) \in E_r, \gamma(1) = x
    \right\}\right).
\end{equation*}
If for some  $x\in Z$ there do not exist rectifiable curves $\gamma \not\in \Gamma_{0}$ with $\gamma(0) \in E_r
$, $\gamma(1) = x$, we set $v(x)=\infty$. Moreover, by construction, $v \equiv 0$ in $E_r$. The measurability of $v$ can be argued as in \cite[Corollary 1.10]{Ja:Ja:Ro:Ro:Sha:07}; this is the only stage where we use the local completeness of $Z$.

By \eqref{eq:negligibleset}, the definition of $\Gamma(E_r,F_r;C')$, and \eqref{eq:pointwiseinequality}, we have
\begin{equation}\label{eq:smallenough:application}
    | v(y) - v(x) |
    =
    v(y)
    \geq
    C' b(r)
    \geq
    \dfrac{C' d(x,y)}{1+2\lambda_r}
    >
    C d(x,y)
    \quad\text{for every $(x,y) \in E_r \times F_r$}.
\end{equation}
Consider next the bounded $1$-Lipschitz function $w(x) = \max\left\{ 0, ( 2 + C' ) b(r) - d(x,E_r) \right\}$.

Denote $$u(x) = \min\left\{ w(x), v(x) \right\}.$$ We claim that $u\in W^{1,\infty}(Z)$. We make three observations. First, $u$ is measurable as a minimum of two measurable functions. Second, the fact $w\in L^{\infty}(Z)$ implies $u\in L^{\infty}(Z)$. And third, $\chi_{Z}\in L^{\infty}(Z)$ is a weak upper gradient of $u$. Indeed, by using \cite[Proposition 6.3.23]{HKST2015}, it is enough to prove that $\chi_{Z}$ is a weak upper gradient for both $w$ and $v$. As $w$ is $1$-Lipschitz, the conclusion is clear for $w$. For $v$, observe that whenever $\gamma \colon \left[0,1\right] \rightarrow Z$ is a constant speed curve and $\gamma \not\in \Gamma_0$ with $v( \gamma(0) ) < \infty$, then $v( \gamma(1) ) < \infty$ and
\begin{equation*}
    | v( \gamma(1) ) - v( \gamma(0) ) | \leq \int_{ \gamma } \chi_{Z} \,ds.
\end{equation*}
The latter inequality follows straight-forwardly from the definition of $v$ and $\Gamma_0$.

We have proved that $u\in \widehat{N}^{1,\infty}(Z)$ with upper gradient $\chi_{Z}$, and thus, by assumption, $z \mapsto 2^{-1}C$ is a Haj\l{}asz gradient of $u$ up to scale $R$. Notice that, by \eqref{eq:smallenough:application} and $w(y) \geq (1-\lambda_r+C')b(r) \geq C' b(r)$ for $y\in  F_r$, we have
\begin{equation*}
    | u(x) - u(y) |
    =
    | u(y) |
    \geq 
    C' b(r)
    >
    C d(x,y)
    \quad\text{for every $(x,y) \in E_r \times F_r$.}
\end{equation*}
This contradicts the fact that $2^{-1}C$ is a Haj\l{}asz gradient of $u$ up to scale $R$. So the claim follows.
\end{proof}

The next step is to show that if we have a doubling measure up to some scale, then the weak $(1,\infty)$-Poincaré inequality implies the $\infty$-thick $(C,R)$-quasiconvexity. The proof of this fact is nowadays quite standard for experts in the field (see for instance \cite[Proposition 4.3]{Durand_2}), so we omit the details. 

\begin{proposition}\label{prop:quasiconvexity_2}
    Let $Z$ be a metric measure space that is locally complete and doubling up to scale $r_1>0$ with constant $C_{Z}$. Suppose that $Z$ supports a weak $(1,\infty)$-Poincaré inequality up to some scale $r_2>0$ with constants $C>0$, $\lambda\geq 1$. Then $Z$ is $\infty$-thick $(C',R)$-quasiconvex for constants $C'>1$ and $R>0$ depending on $r_1$, $r_2$, $C$,  $\lambda$ and $C_Z$.
\end{proposition}

Related to \Cref{prop:quasiconvexity_2}, it is not clear to us if we can dispense with the doubling assumption on $Z$. We formulate the problem as the following special case.

\begin{question}\label{ques:PI}
If a domain $\Omega\subset\R^n$ supports a $(1,\infty)$-weak Poincaré inequality, does it follow that $\Omega$ is quasiconvex?
\end{question}

To solve this problem positively, a natural attempt is to try to deform the Lebesgue measure in $\Omega$ to a doubling measure, mutually absolutely continuous with respect to the Lebesgue measure, for which the $\infty$-weak Poincaré inequality still holds, and then use \Cref{prop:quasiconvexity_2}. However, this problem seems subtle since, according to \cite{Saks:99}, there are (Jordan) domains $\Omega \subset \mathbb{R}^{n}$ which do not support any doubling measure. This implies, in particular, that the proof method of \Cref{prop:quasiconvexity_2} cannot be used without finer analysis.

Finally, we provide the proof of \Cref{thm:equivalence}.
\begin{proof}[Proof of \Cref{thm:equivalence}]
    Recall that  $M^{1,\infty}(Z)$ is $2$-isomorphic to $\LIP^{\infty}(Z)$.

    $(2)\Rightarrow (1)$: We are assuming that $Z$ is very $\infty$-thick $(C,R)$-quasiconvex, hence $\infty$-thick $(C,R)$-quasiconvex. By using \Cref{prop:thick:to:Hajlasz} and \Cref{lemm:local:to:global}, we conclude that $M^{1,\infty}(Z)=W^{1,\infty}(Z)$ and that
    \begin{align*}
       \| u \|_{ M^{1,\infty}(Z) }
       &\leq
       \|u\|_{ L^{\infty}(Z) } + \max\left\{ C \mathscr{E}_{\infty}(u), R^{-1}\|u\|_{L^{\infty}(Z)} \right\}
       \\
       &\leq
       \max \{C, 1+R^{-1}\}
       \| u \|_{  W^{1,\infty}(Z) }
   \end{align*}
    for every $u \in W^{1,\infty}(Z)$. Recall that $\mathscr{E}_{\infty}(u)$ denotes the Sobolev energy of $u$. This implies the norm equivalence.

    $(1)\Rightarrow (2)$: Assume $M^{1,\infty}(Z)=W^{1,\infty}(Z) $. Thanks to \Cref{thm:very_thick-iff-infty-thick}, it is enough to prove that $Z$ is $\infty$-thick $(C,R)$-quasiconvex for some $C\geq 1$, $R>0$. For that we need a slight modification of the proof of \Cref{prop:Hajlasz:to:thick}. We provide the details. 
      
    From $M^{1,\infty}(Z)=W^{1,\infty}(Z)$, it follows that there exists a constant $A \geq 1$ (cf. to $c^{-1}$ in \eqref{eq:equivalenceofnorms}) such that every $u \in W^{1,\infty}( Z )$ has a Lipschitz representative $\widehat u$ satisfying $\| \widehat u \|_{ \LIP^{\infty}(Z) } \leq A \| u \|_{ W^{1,\infty}(Z) }$. If $A = 1$, we consider $R > 0$ and if $A > 1$, we consider $R \in ( 0, 1/( 2(A-1) ) )$. We denote
    $$C_0=\begin{cases}
        1, & \text{if $A = 1$ and $R > 0$, and}\\
        \frac{2(A-1)2R+A}{1-2R(A-1)}, &\text{if $A > 1$ and $(2R)(A-1) < 1$}.
    \end{cases}$$
    If $\epsilon > 0$, there exist $C \in (C_0, C_0+\epsilon)$ and $C' \in ( C, C_0 + \epsilon )$ for which
    \begin{align*}
        C > ( A-1 )(2R)( 2 + C' ) + A.
    \end{align*}
    We prove by contradiction that $Z$ is $\infty$-thick $(C',R)$-quasiconvex. Suppose that there exist $x_0 \in Z$ and measurable sets $E, F \subset B( x_0, R )$ such that $\mu(E)>0$, $\mu(F) > 0$ and that $\Gamma( E, F; C' )$ is $\infty$-negligible. We construct $v$, $w$, and $u$ as in the proof of \Cref{prop:Hajlasz:to:thick}, by using the above radius $R$ and constants $C$ and $C'$. Then the constructed element $u \in W^{1,\infty}(Z)$ has a Lipschitz representative $\widehat{u}$ and the construction guarantees that $\LIP(\widehat{u}) > C$ and that $\| u \|_{L^{\infty}(Z)} \leq \| w \|_{L^{\infty}(Z)} \leq (2R)(2+C')$. Therefore
    \begin{align*}
        C
        &<\LIP(\widehat{u})=
        \| \widehat{u} \|_{ \LIP^{\infty}( Z ) } - \|u\|_{ L^{\infty}(Z) }
        \leq
        (A-1) \| u \|_{ L^{\infty}(Z) } + A \|\rho_{u}\|_{ L^{\infty}(Z) }
        \\
        &\leq
        ( A - 1 ) (2R)(2+C') + A.
    \end{align*}
    This is a contradiction. Therefore $Z$ is $\infty$-thick $(C',R)$-quasiconvex, where $C'>C_0$ is arbitrary. Recalling \Cref{cor:thick-to-very-thick}, this implies that $Z$ is very $\infty$-thick $(C',R)$-quasiconvex for every $C' > C_0$.

Let us prove that $(2)$ implies that $Z$ supports a weak $(1,\infty)$-Poincaré inequality up to some scale $r>0$. To this end, if $C' > C$ and $r \in (0,R]$ such that $Z$ is $\infty$-thick $(C,R)$-quasiconvex and $u \in W^{1,\infty}(Z)$, it follows from \Cref{prop:thick:to:Hajlasz} that $z \mapsto C \| \rho_u \|_{ L^{\infty}( B(z,C'r) ) }$ is a Haj\l{}asz gradient of $u$ up to scale $r$ and thus
\begin{align*}
    \aint{ B(x,r) } | u - u_{ B(x,r) } |\,d\mu
    &\leq
    \aint{ B(x,r) } \aint{ B(x,r) } | u(z)-u(y) |\,d\mu(y)\,d\mu(z)
    \\
    &\leq
    2C \diam( B(x,r) )\| \rho_u \|_{ L^{\infty}(B(x,C'r)) }.
\end{align*}
To conclude the proof of \Cref{thm:equivalence}, we need to show that under a doubling measure up to some scale, the weak $(1,\infty)$-Poincaré inequality implies $(2)$. This follows from \Cref{prop:quasiconvexity_2} and  \Cref{thm:very_thick-iff-infty-thick}.
\end{proof}

\begin{remark}\label{rem:asymptoticbehaviour}
We now give some comments on the asymptotic behavior of the constants $A\in [1,\infty)$ and $(C,R)\subset [1,\infty)\times (0,\infty]$ in the proof of \Cref{thm:equivalence}. 
  \begin{itemize}
      \item $M^{1,\infty}(Z) \neq W^{1,\infty}(Z)$ if and only if $Z$ is not $\infty$-thick $(C,R)$-quasiconvex for any $C \geq 1$, $R > 0$.
      \item $A = 1$ if and only if $Z$ is very $\infty$-thick $C$-quasiconvexity for every $C>1$. This proves \Cref{prop:sharpness}.
      \item If the optimal constant $A$ satisfies $A\to\infty$, any constants for which $\infty$-thick $(C,R)$-quasiconvexity holds satisfy $R\to 0$ and $C\to \infty$.
      \item If $A\to 1^+$, there exist $R=R(A)\in (0, 1/(2(A-1)))$ with $R\to\infty$ and $2R(A-1)\to 0^+$. In particular, we obtain very $\infty$-thick $(C,R)$-quasiconvexity with $R\to \infty$ and $C\to 1^+$.
  \end{itemize}
\end{remark}

Next, we prove \Cref{thm:bounded:doubling}.
\begin{proof}[Proof of \Cref{thm:bounded:doubling}]
We will prove the equivalence with the chain of implications (2) $\Rightarrow$ (3) $\Rightarrow$ (1) $\Rightarrow$ (2). By \Cref{lemm:Lip:Haj}, it suffices to prove the claim for $M^{1,\infty}(Z)$ in place of $\LIP^{\infty}(Z)$. Now, by assumption (2), we have that $Z$ is very $\infty$-thick $(C,R)$-quasiconvex for every $R > 0$. It follows from \Cref{thm:very_thick-iff-infty-thick} and \Cref{prop:thick:to:Hajlasz} that if $u \in W^{1,\infty}(Z)$ has an $\infty$-weak upper gradient $\rho$, then $2^{-1}C \| \rho \|_{ L^{\infty}(Z) }$ is a Haj{\l}asz gradient of $u$ up to scale $R$ for every $R > 0$. Thus $2^{-1}C\| \rho \|_{ L^{\infty}(Z) }$ is a Haj{\l}asz gradient of $u$. By the arbitrariness of $u$, the equality $M^{1,\infty}(Z) = W^{1,\infty}(Z)$ as sets and the equivalence of the energy seminorms follows.

Under assumption (3), let $C$ be such that
\begin{align*}
    2^{-1}C
    >
    \sup
    \left\{
        \frac{ \mathcal{E}_{\infty}^{\infty}(u) }{ \mathscr{E}_{\infty}(u) }
        \colon
        u \in W^{1,\infty}(Z), \mathscr{E}_{\infty}(u) > 0
    \right\};
\end{align*}
the existence of $C$ follows from the equivalence of the seminorms. Therefore the constant function $z \mapsto 2^{-1}C\mathscr{E}_{\infty}(u)$ is a Haj{\l}asz gradient of $u$ for every $u \in W^{1,\infty}(Z)$.

We conclude from \Cref{prop:Hajlasz:to:thick} that $Z$ is $\infty$-thick $( C', R )$-quasiconvex for every $C' > C$ and $R > 0$, and, by \Cref{cor:thick-to-very-thick}, it follows that $Z$ is very $\infty$-thick $( C', R )$-quasiconvex for every $C' > C$ and $R > 0$. The equality $M^{1,\infty}(Z) = W^{1,\infty}(Z)$ and the $C'$-quasiconvexity of $Z$ follow, so (1) holds.

Under assumption (1), we have that $\LIP^{\infty}(Z) = W^{1,\infty}(Z)$ and $Z$ is $\widetilde{C}$-quasiconvex for some $\widetilde{C} \geq 1$. \Cref{thm:equivalence} implies that $Z$ is very $\infty$-thick $( C, R )$-quasiconvex for some $C\geq 1$ and $R > 0$. It follows that $Z$ is very $\infty$-thick $C\widetilde{C}$-quasiconvex by \Cref{lem:thick-to-verythick}.

The validity of the weak $\infty$-Poincaré inequality follows from the very $\infty$-thick quasiconvexity and the argument at the end of the proof of \Cref{thm:equivalence}.
\end{proof}

\begin{proof}[Proof of \Cref{cor:bounded:doubling:PI}]
It is clear that the weak $(1,\infty)$-Poincaré inequality implies the weak $(1,\infty)$-Poincaré inequality up to some scale. For the converse direction, we first use \Cref{thm:equivalence}, which yields $\text{LIP}^{\infty}(Z)=W^{1,\infty}(Z)$. Since, by assumption, the space $Z$ is quasiconvex, we may apply \Cref{thm:bounded:doubling} to conclude that $Z$ supports a weak $(1,\infty)$-Poincaré inequality.
\end{proof}

\section{Essential distance and infinity harmonic functions}\label{sec:essentialdistance:proofs}

\subsection{Essential distance}
The geometric property in \Cref{prop:sharpness}  characterizing the isometric equality $\LIP^{\infty}(Z)= W^{1,\infty}(Z)$, that is, to be very $\infty$-thick $C$-quasiconvexity for every $C>1$, is equivalent to the \emph{$\infty$-weak Fubini property} introduced in \cite[Definition 4.1 and Proposition 4.2]{Dur:Jar:Sha:19} (see also \cite{Juut:Shan:06}). Basic examples satisfying the assumptions of \Cref{prop:sharpness} include the Euclidean space $\mathbb{R}^n$ and (sub-)Riemannian spaces \cite{Drag:Manf:Vitt:13}; see \cite[Section 5]{Juut:Shan:06} for further discussion.

\Cref{thm:essentialdistance} below shows that examples of metric measure spaces satisfying \Cref{prop:sharpness} can readily be obtained from \Cref{thm:equivalence}. Indeed, we consider for every curve family $\Gamma$ in $Z$ the function
\begin{align*}
    d_{\Gamma}(x,y)
    \coloneqq
    \inf\left\{
        \ell( \gamma ) \colon \gamma(0) = x, \gamma(1) = y, \gamma \colon [0,1] \rightarrow Z \not\in \Gamma
    \right\}
    \quad\text{for $x,y \in Z$.}
\end{align*}
We use the convention $\inf(\emptyset)=+\infty$. Next, we define the \emph{essential distance}
\begin{align}\label{eq:essentialdistance}
    \widehat{d}( x, y )
    \coloneqq
    \sup\left\{
        d_{\Gamma}(x,y)
        \colon
        \text{ $\Gamma$ is $\infty$-negligible}
    \right\}
    \quad\text{for $x,y \in Z$.}
\end{align}
This is an equivalent adaptation of the distance of Durand-Cartagena, Jaramillo, and Shanmugalingam from \cite[Equation (9)]{Dur:Jar:Sha:19}. Note that while $\widehat{d}$ is nonnegative, symmetric, and satisfies the triangle inequality, $\widehat{d}$ can very well be infinite. This holds, for instance, for the Sierpi\'nski carpet since every curve family is $\infty$-negligible by the equality $W^{1,\infty}(Z) = L^{\infty}(Z)$. Nevertheless, it always holds that $d(x,y) \leq \widehat{d}(x,y)$. The reverse inequality is strongly connected to \Cref{thm:equivalence}.

\begin{theorem}\label{thm:essentialdistance}
Let $Z$ be a locally complete metric measure space with a reference measure that is infinitesimally doubling, and finite and positive on all balls, and that is very $\infty$-thick $(C,R)$-quasiconvex. Then every pair with $d(x,y) < R$ satisfies $\widehat{d}(x,y) \leq C d(x,y)$. Moreover, $\widehat{d}$ restricts to a finite length metric on each connected component $U$ of $Z$, the Sobolev spaces $W^{1,\infty}$ in every open subset $\Omega \subset U$ with respect to $d$ and $\widehat{d}$, respectively, are isometric, and the metric measure space $( U, \widehat{d}, \mu )$ is very $\infty$-thick $C$-quasiconvex for every $C>1$. Moreover, if $Z$ is complete and locally compact, then $( U, \widehat{d} )$ is a proper length space.
\end{theorem}

\begin{proof}
It is immediate from very $\infty$-thick $(C,R)$-quasiconvexity that $d_{\Gamma}(x,y) \leq Cd(x,y)$ whenever $\Gamma$ is $\infty$-negligible and $d(x,y) < R$. The inequality $\widehat{d}(x,y) \leq C d(x,y)$ follows. For the rest of the proof, we adapt an argument from the proof of \cite[Proposition 4.4]{Dur:Jar:Sha:19}. Given $x, y \in Z$, we claim that there exists an $\infty$-negligible $\Gamma_{x,y}$ so that
\begin{align*}
    \widehat{d}(x,y) = d_{ \Gamma_{x,y} }(x,y).
\end{align*}
Indeed, if $\widehat{d}(x,y) = \infty$, for every $n\in \mathbb{N}$, there is an $\infty$-negligible $\Gamma_n$ so that $d_{ \Gamma_n }(x,y) \geq n$. Therefore $\Gamma_{x,y} \coloneqq \bigcup_{ n \geq 1 } \Gamma_{n}$ satisfies
\begin{align*}
    n \leq d_{ \Gamma_n }(x,y) \leq d_{ \Gamma_{x,y} }(x,y) \leq \widehat{d}(x,y),
\end{align*}
so the claim follows by passing to the limit $n \rightarrow \infty$. If $\widehat{d}(x,y) < \infty$, we instead find $\infty$-negligible $\Gamma_n$ so that $| d_{\Gamma_n}(x,y) - \widehat{d}(x,y) | < n^{-1}$ and defining $\Gamma_{x,y}$ as above yields that
\begin{align*}
    \widehat{d}(x,y) - \frac{1}{n} \leq d_{ \Gamma_n }(x,y) \leq  d_{ \Gamma_{x,y} }(x,y) \leq \widehat{d}(x,y),
\end{align*}
so the conclusion follows by passing to the limit $n \rightarrow \infty$.

Next, consider a countable and dense set $D \subset Z$. Consider the $\infty$-negligible family $\Gamma \coloneqq \bigcup_{ x,y \in D, x \neq y } \Gamma_{x,y}$. Recall that there exists a negligible Borel set $B \subset Z$ such that $\Gamma \subset \Gamma^{+}_{B}$. In particular, $\Gamma^{+}_{B}$ is $\infty$-negligible. It follows from the construction of $\Gamma$ that
\begin{align*}
    \widehat{d}(x,y) = d_{ \Gamma_{x,y} }(x,y) \leq d_{\Gamma^{+}_{B}}(x,y) \leq \widehat{d}(x,y)
    \quad\text{for $x, y \in D$, $x \neq y$,}
\end{align*}
so $\widehat{d}(x,y) = d_{ \Gamma^{+}_{B} }(x,y)$ for every $x, y \in D$. We extend this equality for every $x, y \in Z$.

Given $\epsilon > 0$, we find $x', y' \in D$ such that
\begin{equation}\label{eq:errorcontrol}
    \max\{ \widehat{d}(x,x'), \widehat{d}( y, y' ) \} < \frac{ \epsilon }{ 4 }.
\end{equation}
Inequality \eqref{eq:errorcontrol} and the triangle inequality for $\widehat{d}$ imply that
\begin{align*}
    \widehat{d}(x,y) \leq \frac{ \epsilon }{ 2 } + \widehat{d}(x',y')
    \quad\text{and}\quad
    \widehat{d}(x',y') \leq \frac{ \epsilon }{ 2 } + \widehat{d}(x,y);
\end{align*}
similar inequality holds for $d_{ \Gamma^{+}_{B} }$. Since $\widehat{d}(x',y') = d_{ \Gamma^{+}_{B} }(x',y')$, it follows that $\widehat{d}(x,y) < \infty$ if and only if $d_{ \Gamma^{+}_{B} }(x, y) < \infty$ if and only if $d_{ \Gamma^{+}_{B} }(x',y') < \infty$. Moreover, in case $\widehat{d}(x,y) < \infty$, it holds that
\begin{align*}
    | \widehat{d}(x,y) - d_{\Gamma^{+}_{B}}(x,y) | < \epsilon.
\end{align*}
Since $\epsilon > 0$ is arbitrary, we conclude that
$$\widehat{d}(x,y) = d_{\Gamma^{+}_{B}}(x,y)\quad\text{for every $x, y \in Z$}.$$
This equality implies that $\widehat{d}$ restricts to a finite length distance on the connected components of $Z$. Indeed, $\widehat{d}$ restricts to a finite distance that is bi-Lipschitz comparable to $d$ for every metric ball of radius $R$ with respect to $d$. Hence the topologies induced by $\widehat{d}$ and $d$ are the same. Then a standard connectivity argument implies that $\widehat{d}$ is finite in each connected component of $Z$.

Regarding the Sobolev spaces with respect to $d$ and $\widehat{d}$ being equal, observe that for every curve it holds that 
$$\ell_{d}( \gamma ) \leq \ell_{ \widehat{d} }( \gamma ) \leq C \ell_{ d }( \gamma ),$$ 
so the classes of rectifiable curves with respect to $d$ and $\widehat{d}$ coincide. Since $\widehat{d} = d_{ \Gamma^{+}_{B} }$, it also readily follows that if a rectifiable curve is not in $\Gamma^{+}_{B}$, then the length measures of the curve with respect to $d$ and $\widehat{d}$ coincide. In particular, for every $\gamma\notin \Gamma^{+}_{B}$ and every Borel function $\rho \colon Z \to [0,\infty]$, the values of the path integrals $\int_{\gamma}\rho\, ds$ with respect to $d$ and $\widehat{d}$, respectively, agree. Since $\Gamma^{+}_{B}$ is $\infty$-negligible, it follows that the classes of $\infty$-weak upper gradients coincide for the distances $d$ and $\widehat d$, and hence the Sobolev spaces are isometrically isomorphic in every domain $\Omega \subset Z$. This observation also implies that the restriction of $\widehat{d}$ to a connected component of $U$ is very $\infty$-thick $C$-quasiconvex for every $C > 1$.

It remains to verify that if $U \subset Z$ is a connected component and $Z$ is complete and locally compact, then $( U, \widehat{d} )$ is proper. To this end, since $( U, \widehat{d} )$ is a length space, it suffices to prove that $( U, \widehat{d} )$ is complete and locally compact by the Hopf--Rinow theorem, cf. \cite[Theorem 2.5.28]{Bu:Bu:Iv:01}. The completeness follows because $( U, d )$ is complete and the Cauchy sequences with respect to $d$ and $\widehat{d}$ coincide, while local compactness follows from the equivalence of the topologies induced by $d$ and $\widehat{d}$.
\end{proof}

\subsection{Infinity harmonic functions}\label{sec:infinitharmonicfunctions}
In this section, we prove \Cref{thm:infty-harmonic-solution}.
\begin{proof}[Proof of \Cref{thm:infty-harmonic-solution}]
During the proof, we refer to a length space which is very $\infty$-thick $C$-quasiconvex for every $C>1$ as having the $\infty$-weak Fubini property. This is equivalent to the definition used in \cite[Definition 4.1 and Proposition 4.2]{Dur:Jar:Sha:19}.

By applying \Cref{thm:essentialdistance}, the Sobolev space in a connected component $U \subset Z$ containing $\Omega$ isometrically coincides with the one defined by the essential distance $\widehat{d}$. Indeed, $W^{1,\infty}( V, \widehat{d}, \mu )=W^{1,\infty}(V,d,\mu) $ isometrically for every open set $V \subset U$. It also follows from \Cref{thm:essentialdistance} and the precompactness of $\Omega \subset Z$ that $\Omega$ is bounded in $( U, \widehat{d} )$. The claim follows from the strategy of proof of \cite[Theorem 1.2]{Dur:Jar:Sha:19} that establishes the equivalence of $\infty$-harmonicity in $( U, d, \mu )$ to being AMLE in $(U, \widehat{d} )$.

We first note that \cite[Theorem 1.2]{Dur:Jar:Sha:19} assumes that the measure $\mu$ is doubling. However, this is only used during the proof to establish the properness of $( U, \widehat{d}, \mu )$ which we have by \Cref{thm:essentialdistance}. Thus the proof needs only cosmetic changes. We sketch the main ideas. The authors use the basic properties of complete length spaces satisfying the $\infty$-weak Fubini property to prove that $\infty$-harmonic functions are equivalent to the so-called \emph{strong-AMLEs} by \cite[Lemma 4.6 and Remark 4.7]{Dur:Jar:Sha:19}. This argument does not need the doubling property or the properness. The equivalence of strong-AMLEs and AMLEs use the properness and the $\infty$-weak Fubini property, together with \cite[Propositions 4.1 and 5.8]{Juut:Shan:06} by Juutinen and Shanmugalingam. The first proposition holds in length spaces while the latter property uses the properness and a slightly stronger version of the $\infty$-weak Fubini property. However, as remarked in \cite[Proof of Theorem 4.10]{Dur:Jar:Sha:19}, the proof of \cite[Proposition 5.8]{Juut:Shan:06} goes through with the $\infty$-weak Fubini property. 

To obtain the existence and uniqueness, let $g \in \LIP^{\infty}(Z)$ and consider $g' = g|_{U}$. Then $g' \in \LIP^{\infty}(U,d) = W^{1,\infty}(U,d,\mu) = W^{1,\infty}( U, \widehat{d}, \mu ) = \LIP^{\infty}( U, \widehat{d} )$. By the results by Peres, Schramm, Sheffield and Wilson for all length spaces in \cite[Theorem 1.4]{Per:Schr:Sheff:Wil:09}, we obtain a unique AMLE extension $h$ of $g'|_{ U \setminus \Omega }$. In particular, by the comments above, $h$ is $\infty$-harmonic in $\Omega$. Since $\Omega$ is bounded with respect to $\widehat{d}$, it follows that $h \in \LIP^{\infty}(U, \widehat{d})¨$ and thus $h \in \LIP^{\infty}(U,d)$. If we extend $h$ to be equal to $g$ in $Z \setminus U$, we obtain $u \in \LIP^{\infty}(Z,d)$ that is $\infty$-harmonic in $\Omega$ coinciding with $g$ in $Z \setminus \Omega$. The uniqueness of $u|_{\Omega}$ follows from the AMLE uniqueness in \cite{Per:Schr:Sheff:Wil:09} and the $\infty$-harmonic and AMLE correspondence above.
\end{proof}

\subsection{Examples of infinity harmonicity and AMLEs}
The following example, based on \cite{Aron:Cran:Juu:04}, illustrates that even in simple cases, Lipschitz functions admit $\infty$-harmonic extensions which are not AMLE. This shows why the essential distance is important for the proof strategy of \Cref{thm:infty-harmonic-solution}.
\begin{example}\label{ex:necessityofrenorming}
{\em In case $Z = [0,1]\times\{0\} \cup \{0\} \times [0,1] \subset \mathbb{R}^2$ is equipped with the length measure and the Euclidean distance, a nonconstant map $h \colon \{ (1,0) \} \cup \{(0,1)\} \rightarrow \mathbb{R}$ does not have an AMLE extension, see e.g. the discussion after \cite[Theorem 1.4]{Per:Schr:Sheff:Wil:09}. However, since $Z$ is very $\infty$-thick $\sqrt{2}$-quasiconvex, $h$ admits an $\infty$-harmonic extension by \cite{Dur:Jar:Sha:19} or \Cref{thm:infty-harmonic-solution}. Notice that $Z$ is a $1$-PI space. See \cite[Example 4.8]{Dur:Jar:Sha:19} for a related example.}
\end{example}

The opposite can happen: the length distance in the Sierpi\'nski carpet is such that Lipschitz functions have a unique AMLE by \cite{Juut:02,Per:Schr:Sheff:Wil:09} but the $\infty$-harmonic extension problem is highly non-unique.
\begin{example}\label{ex:renormingnotenough}
{\em Suppose that $Z$ is the Sierpi\'nski carpet equipped with the standard self-similar measure and $\Omega \neq Z$ is a domain. Then, for a given Lipschitz map $g \colon Z \rightarrow \mathbb{R}$, any Lipschitz function $h \colon Z\to\mathbb R$ so that $g\equiv h$ on $Z\setminus\Omega$ acts as an $\infty$-harmonic extension of $g$ in $\Omega$. Indeed, since $W^{1,\infty}(Z) = L^{\infty}(Z)$, the Sobolev energy of any Lipschitz function in $\Omega$ is zero. Note that even if we consider the bi-Lipschitz equivalent length distance in $Z$, the same conclusion holds even though existence and uniqueness of AMLEs holds by \cite{Per:Schr:Sheff:Wil:09}.}
\end{example}

\section{Sobolev extension sets}\label{sec:extension}

The goal of this section is to prove \Cref{thm:Main_inf.ext.} and to discuss basic properties of Sobolev extension sets.
\begin{definition}\label{def:sob_ext_set}
Let $Z$ be a metric measure space and $1\leq p\leq \infty$. A subset $\Omega\subset Z$ is a $W^{1,p}$-extension set, for some $1\leq p\leq \infty$, if there exists an extension operator (not necessarily linear) $E\colon W^{1,p}(\Omega)\to W^{1,p}(Z)$ so that $Eu|_{\Omega}=u$ and $\|Eu\|_{W^{1,p}(Z)}\leq C\|u\|_{W^{1,p}(\Omega)}$ for every $u\in W^{1,p}(\Omega)$ for a constant $C\geq 1$ independent of $u$.
\end{definition}
There exists a vast literature treating the topic of Sobolev extension sets (both in the classical Euclidean setting and metric measure setting) and we refer the interested reader to \cite{HKT2008:B,HKT2008,Garc:Iko:Zhu:23} and references therein for further reading on the topic.

The measure density condition is a measurability condition that plays an important role in the theory Sobolev extension sets. In particular, if $1\leq p<\infty$, every $W^{1,p}$-extension domain $\Omega \subset \mathbb{R}^{n}$ satisfies the measure density condition (see \cite{HKT2008}).

\begin{definition}
    Given a metric measure space $Z$ we say that a measurable subset $\Omega\subset Z$ satisfies the measure density condition if there exists $C_{\mu}>0$ so that for every $x\in\overline{\Omega}$,
    $$\mu (B(x,r))\leq C_\mu \mu(\Omega\cap B(x,r)),\quad \text{for every } r\in (0,\min\{\diam(\Omega),1\}) .$$
\end{definition}

If $p=\infty$, there exist $W^{1,\infty}$-extension sets that do not satisfy a measure density condition. The following is a simple example.
\begin{example}\label{ex:1}
{\em Let $C\subset [0,1]$ be a fat Cantor set with positive measure. Almost every point of $C$ is of density $1$ on $C$, so $[0,1]\setminus C$, whose closure is the whole interval $[0,1]$, cannot satisfy the measure density condition. Now consider $\Omega =\mathbb{R}^n\setminus C^n$. By using \cite[Theorem A]{HH2008}, it follows that $\Omega$ is quasiconvex. Consequently, by using \Cref{prop:HKT2008}, we have that $\Omega$ is a $W^{1,\infty}$-extension domain. Moreover, by Fubini's theorem, we have that $\mathcal{L}^{n}(C^n) > 0$ and since $C^n=\partial\Omega$, the domain $\Omega$ fails the measure-density condition.}
\end{example}
There are examples of $W^{1,\infty}$-extension domains for which the restriction of the Lebesgue measure is not doubling.
\begin{example}\label{ex:2}
{\em The outward cusp in \Cref{carpet:outwrd_cusp_2D} is defined as
    \begin{equation}\label{eq:outward_cusp}
\Omega_{t^2}=B\left((3,0),\sqrt{5}\right)\cup\{(x,y)\in\R^2\colon  x\in (0,1), |y|<x^2\} .\end{equation}

\begin{figure}
\begin{center}
\begin{tikzpicture}
\draw[line width=0.8 pt, scale=0.8, domain=0:1, smooth, variable=\x, black!100] plot ({\x}, {\x*\x});
\draw[line width=0.8 pt, scale=0.8, domain=0:1, smooth, variable=\x, black!100] plot ({\x}, {-\x*\x});
\draw [line width=0.8 pt, scale=0.8, black!100] (5.236,0) arc [start angle=0, end angle=-154, radius=2.236cm];
\draw [line width=0.8 pt, scale=0.8, black!100] (5.236,0) arc [start angle=0, end angle=154, radius=2.236cm];
\end{tikzpicture}
\end{center}
\caption{Outward cusp $\Omega_{t^2}$}
   \label{carpet:outwrd_cusp_2D}
\end{figure}
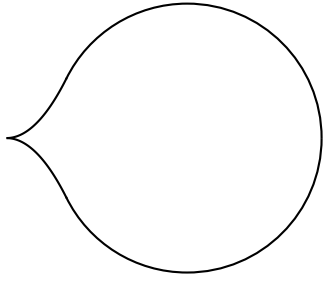

The set $\Omega_{t^2}$ is a $W^{1,\infty}-$extension domain according to \Cref{prop:HKT2008}. It is not difficult to prove that the measure density condition fails at the cusp $(0,0)$, yet the restriction of the Lebesgue measure $\mathcal L^n$ to $\Omega$ is doubling.

By replacing $t^2\colon (0, 1]\to(0, \infty)$ with a left-continuous and increasing function $\psi\colon (0, 1]\to(0, \infty)$ with $\psi(1)=1$, we obtain a cusp domain
\begin{equation*}
\Omega_{\psi}= B\left((3,0),\sqrt{5}\right) \cup \{(x,y)\in\R^2\colon x\in (0,1), |y|<\psi(x)\}.
\end{equation*}
similar to $\Omega_{t^2}$. If we choose such $\psi$ with 
\begin{equation}\label{eq:nondoubling}
\lim_{t\to 0^+}\frac{\psi(2t)}{\psi(t)}=\infty,
\end{equation}
then the Lebesgue measure restricted to $\Omega_{\psi}$ is not doubling up to any scale. Since it is still quasiconvex, it is a $W^{1, \infty}$-extension domain.}
\end{example}
We note that both \Cref{ex:1,ex:2} satisfy the weak $(1,\infty)$-Poincaré inequality by \Cref{thm:equivalence}, \Cref{thm:Main_inf.ext.}, and \Cref{cor:bounded:doubling:PI}. \Cref{ex:2} provides a further example of a metric measure space satisfying the assumptions of \Cref{thm:equivalence} but whose measure is not doubling and thus not being an $\infty$-PI space.

We are ready to prove \Cref{thm:Main_inf.ext.}.
\begin{proof}[Proof of \Cref{thm:Main_inf.ext.}]
For the equivalence between $(2)$ and $(3)$, we just need to apply \Cref{lem:thick-to-verythick}. For the equivalence $(1)\Leftrightarrow (2)$, we can relax the assumptions. Indeed, we only need the weaker assumptions that $\Omega\subset Z$ is locally complete and infinitesimally doubling while we may relax the assumptions of $Z$ to the equality $\LIP^{\infty}(Z) = W^{1,\infty}(Z)$ without assuming infinitesimally doubling or very thick quasiconvexity.

$(2)\Rightarrow (1)$: We let $T\colon\LIP(\Omega)\to\LIP(Z)$ be the well-known McShane extension operator:
$$ T u(x)= \inf\{ u(y)+\lip(u) d(x,y)\colon y\in\Omega\}  \quad  x\in Z,\; u\in \LIP(\Omega).$$
The extension satisfies $\lip(Tu)=\lip(u)$ for every $u\in \LIP(\Omega)$. Next, we denote $Eu(x) = \max\left\{ -\|u\|_{L^{\infty}(\Omega)}, \min\left\{ \|u\|_{L^{\infty}(\Omega)}, Tu(x) \right\} \right\}$ for $x \in Z$. Then $$\| Eu \|_{L^{\infty}(Z)} = \|u\|_{L^{\infty}(Z)} \quad \text{and} \quad \lip(Eu) = \lip(u).$$ The map $u \mapsto Eu$ is a norm-preserving extension operator from $\LIP^{\infty}(\Omega)$ into $\LIP^{\infty}(Z)$. By applying \Cref{thm:equivalence} to $\Omega$, we have that $\LIP^{\infty}(\Omega) = W^{1,\infty}(\Omega)$ with equivalent norms. Since also $\LIP^{\infty}(Z) = W^{1,\infty}(Z)$ with equivalent norms, here $E$ defines a bounded extension operator from $W^{1,\infty}(\Omega)$ into $W^{1,\infty}(Z)$.

$(1)\Rightarrow (2)$: Given $u \in W^{1,\infty}(\Omega)$, we know that there exists $Eu \in W^{1,\infty}(Z)$ extending $u$, and $Eu$ has a representative in $\LIP^{\infty}(Z)$. The restriction of the $\LIP^{\infty}(Z)$ representative to $\Omega$ is a representative of $u$ in $\widehat{N}^{1,\infty}(\Omega)$. Thus $u$ has a representative in $\LIP^{\infty}(\Omega)$. It follows that $\LIP^{\infty}(\Omega) = W^{1,\infty}(\Omega)$, so applying \Cref{thm:equivalence} concludes the proof.
\end{proof}

\begin{remark}\label{rem:linear_ext}
Under mild assumptions, there exists a linear extension operator $E\colon W^{1,\infty}(\Omega)\to W^{1,\infty}(Z)$ in \Cref{thm:Main_inf.ext.}. For example, if one assumes that $Z$ is metrically doubling. In this case, instead of using the (nonlinear) McShane extension theorem one, there exists a (linear) Whitney-type extension operator for Lipschitz and bounded functions (see for instance \cite[Theorem 4.1.21]{HKST2015}). We extend this fact to a sharper condition, that of finite Nagata dimension, in the subsequent section. The connection between Nagata dimension and Lipschitz extensions were considered in \cite{La:Sch:05}.
\end{remark}

\section{Banach-valued Sobolev functions}\label{sec:banach}
In this section, we prove \Cref{prop:Banach-valued_version} and provide the necessary background results.

For a metric measure space $Z$ and a Banach space $\mathbb V=(\mathbb V, |\cdot|)$, the definition of the Lipschitz spaces $\LIP^{\infty}(Z;\mathbb V)$, Haj{\l}asz spaces $M^{1,\infty}(Z;\mathbb V)$ and the Sobolev spaces $W^{1,\infty}(Z;\mathbb V)$ is completely analogous to the one given in \Cref{sec:prelim}, but now for Bochner measurable functions $u\colon Z\to\mathbb V$. In this case the left-hand sides of \eqref{eq:Lips_cond.}, \eqref{eq:Hineq} and \eqref{eq:boundaryinequality} should be understood with the Banach norm $|\cdot|$. Similarly $\mathbb V$-valued $W^{1,\infty}$ extension sets $\Omega\subset Z$ are those ones for which there exists an extension operator $E\colon W^{1,\infty}(\Omega;\mathbb V)\to W^{1,\infty}(Z;\mathbb V)$.

We refer to \cite{HKST2001,HKST2015} for a more detailed explanation of Banach-valued Sobolev functions and its basic properties; see also \cite{Garc:Iko:Zhu:23}. The following lemma states the basic results we need for the purposes of this section. In what follows, the dual space of $\mathbb{V}$ is denoted by $\mathbb{V}^{*}$. That is, $\mathbb{V}^{*}$ is the collection of all continuous linear maps $w \colon \mathbb{V} \rightarrow \mathbb{R}$, where the norm is defined by $|w| \coloneqq \sup_{ |v| \leq 1 } w(v)$.

\begin{lemma}\label{lem:Banach_valued}
Let $Z$ be a metric measure space and let $\mathbb V$ be a Banach space. 
\begin{enumerate}
    \item The inclusion from $\LIP^{\infty}(Z; \mathbb{V} )$ into $M^{1,\infty}(Z;\mathbb V)$ is surjective and a $2$-isomorphism.
    \item For a measurable $u\colon Z \to \mathbb V$, a Borel function $g\colon Z \to [0,\infty]$ is a Haj{\l}asz gradient of $u$ if and only if $g$ is a Haj{\l}asz gradient of $w(u)\colon Z\to \R$ for every $w\in\mathbb V^*$, $|w|\leq 1$.
    \item For a measurable $u\colon Z\to \mathbb V$, a Borel function $\rho \colon Z \to [0,\infty]$ is an $\infty$-weak upper gradient for some representative of $u$ if and only if for every $w\in\mathbb V^*$, $|w|\leq 1$, $\rho$ is an $\infty$-weak upper gradient for some representative of $w(u)$.
    \item The inclusion from $M^{1,\infty}(Z; \mathbb{V} )$ into $W^{1,\infty}(Z; \mathbb{V} )$ is $2$-bi-Lipschitz.
\end{enumerate}
\end{lemma}
\begin{proof}
The proof of (1) is similar to the proof of \Cref{lemm:Lip:Haj}. For (2), see e.g. \cite[Lemma 3.3]{Garc:Iko:Zhu:23}. For (3), see e.g. \cite[Lemma 3.9]{Garc:Iko:Zhu:23}. Claim (4) follows similarly to \Cref{lemma:M:in:W}.
\end{proof}

\begin{proposition}\label{prop:M=W_Banach-valued}
Let $Z$ be a metric measure space and let $\mathbb V$ be any Banach space. Then  $M^{1,\infty}(Z)=W^{1,\infty}(Z)$ if and only if  $M^{1,\infty}(Z;\mathbb V)=W^{1,\infty}(Z;\mathbb V)$.
\end{proposition}
\begin{proof}
Assume first that $M^{1,\infty}(Z)=W^{1,\infty}(Z)$. Thanks to \Cref{lem:Banach_valued} (4), we only need to consider $u\in W^{1,\infty}(Z;\mathbb V)$ and show that $u\in M^{1,\infty}(Z;\mathbb V)$. For any $w\in \mathbb V^*$ with norm one, we have that $w(u)\in W^{1,\infty}(Z)$ and, by our assumption, we have that $w(u)\in M^{1,\infty}(Z)$ and that there exists a constant $C>0$ (independent of $u$ and $w$) so that
$$\|w(u)\|_{M^{1,\infty}(Z)}\leq C\|w(u)\|_{W^{1,\infty}(Z)} \leq C\|u\|_{W^{1,\infty}(Z;\mathbb{V})} .$$
In particular, the constant function $z \mapsto C\|u\|_{W^{1,\infty}(Z;\mathbb{V})}+1$ acts as a Haj{\l}asz gradient of $w(u)$. It only remains to use \Cref{lem:Banach_valued} (2) to conclude that $C\|u\|_{W^{1,\infty}(Z;\mathbb V)}+1$ is a Haj{\l}asz gradient of $u$, and we are done.

Assume next that $M^{1,\infty}(Z;\mathbb V)=W^{1,\infty}(Z;\mathbb V)$ as sets with comparable norms. Consider unit norm $v \in \mathbb{V}$ and $w \in \mathbb{V}^{*}$ such that $w(v) = 1$. Then $i_v \colon \mathbb{R} \rightarrow \mathbb{V}$, $t \mapsto t \cdot v$, defines an isometric embedding and $p_w \colon \mathbb{V} \rightarrow \mathbb{R}$, $u \mapsto w(u)$, a linear $1$-Lipschitz map. It follows from \Cref{lem:Banach_valued} that these induce an isometric embedding $I_v \colon W^{1,\infty}(Z) \rightarrow W^{1,\infty}(Z; \mathbb{V})$ 
and a $1$-Lipschitz linear map $P_w \colon M^{1,\infty}(Z; \mathbb{V}) \rightarrow M^{1,\infty}(Z)$ by pointwise action. That is, $I_v(u)(x)\coloneqq i_v(u(x))=(u(x))v$ and $P_{w}(u)(x)\coloneqq p_w(u(x))=w(u(x))$ for every $x\in Z$. Now if $u \in W^{1,\infty}(Z)$, then $I_v(u) \in W^{1,\infty}(Z; \mathbb{V}) = M^{1,\infty}(Z; \mathbb{V})$, so $u = P_w( I_v(u) ) \in M^{1,\infty}(Z)$. The proof is complete.
\end{proof}
The above proof shows, in particular, that if $M^{1,\infty}(Z)=W^{1,\infty}(Z)$, then $M^{1,\infty}(Z;\mathbb V)=W^{1,\infty}(Z;\mathbb V)$ and
$$\frac{1}{4}\| u  \|_{W^{1,\infty}(Z;\mathbb V)} \leq \|u\|_{M^{1,\infty}(Z;\mathbb V)}\leq C\| u\|_{W^{1,\infty}(Z;\mathbb V)}\quad \text{for every } u\in W^{1,\infty}(Z;\mathbb V)$$
for a constant $C>0$ independent of the Banach spaces $\mathbb V$. It is not clear if the corresponding statement holds when $p \in [1,\infty)$; see \cite[Theorem 1.6]{Garc:Iko:Zhu:23} for a partial result to this effect.

Next we recall the definition of the Nagata dimension (see \cite{La:Sch:05} for further background).
\begin{definition}\label{def:Nag_dim}
    Let $Z$ be a metric space. The Nagata dimension $\dim_N Z$ of $Z$ is the infimum of all integers $n\geq 0$ for which there
exists a constant $c>0$ such that for all $s>0$, there exists a covering $Z=\bigcup_{i\in \mathcal I}D_i$ satisfying $\diam(D_i)\leq cs$ for all $i\in\mathcal I$ and so that for every $D\subset Z$ with $\diam(D)\leq s$, it holds that $\#\{ i\in\mathcal I\colon  D\cap D_i\neq \emptyset \}\leq n+1$.
\end{definition}
\begin{proposition}\label{prop:real-to-Banach}
Let $Z$ be a metric measure space. Let $\Omega \subset Z$ be a complete set with $\LIP^{\infty}(\Omega) = W^{1,\infty}(\Omega)$ such that $\Omega$ or $Z \setminus \Omega$ has finite Nagata dimension. Then there exists a bounded and linear extension operator $E_{ \mathbb{V} } \colon W^{1,\infty}( \Omega; \mathbb{V} ) \rightarrow \LIP^{\infty}(Z; \mathbb{V} )$ for every Banach space $\mathbb{V}$. 
\end{proposition}
\begin{proof}
By \Cref{prop:M=W_Banach-valued}, we have $M^{1,\infty}(\Omega;\mathbb V)=W^{1,\infty}(\Omega;\mathbb V)$ and also by \Cref{lem:Banach_valued} (1), we have $\LIP^{\infty}( U; \mathbb{V} ) = M^{1,\infty}( U; \mathbb{V} )$ where $U \in \{ \Omega, Z \}$. By uniform continuity, every $u \in \LIP^{\infty}(\Omega; \mathbb{V})$ admits a unique extension to $\overline{u} \in \LIP^{\infty}( \overline{\Omega}; \mathbb{V} )$ with equal norm. Thus it is enough to build a linear extension operator
\begin{equation}\label{eq:EXT-LIP-BANACH}
E_{\mathbb V}\colon \LIP^{\infty}(\overline{\Omega};\mathbb V)\to \LIP^{\infty}(Z;\mathbb V).
\end{equation}
The claim is obvious if $\overline{\Omega} = Z$, so we assume that $Z \setminus \overline{\Omega} \neq \emptyset$.

The existence of $E_{\mathbb V}$ follows from the work of Lang and Schlichenmaier \cite{La:Sch:05}. Since a similar extension is considered in the proof of \cite[Theorem 5.2]{La:Sch:05}, we only explain the main ideas. Since the Nagata dimension of a metric space and its metric completion are equal, it follows that either $Z \setminus \overline{\Omega}$ or $\overline{\Omega}$ has finite Nagata dimension. Let $n$ be the Nagata dimension of $Z \setminus \overline{\Omega}$ in the first case while let $(n-1)$ be the Nagata dimension of $\overline{\Omega}$ in the second case. Let also $c > 0$ be the constant obtained from \Cref{def:Nag_dim}. Now, as in \cite[Proof of Theorem 1.5, p. 3652-3653]{La:Sch:05} (resp. \cite[Proof of Theorem 1.6, p. 3653-3654]{La:Sch:05}), we conclude the existence of a (Whitney-type) covering $( B_i )_{i \in I}$ of $Z \setminus \overline{\Omega}$  by non-empty subsets of $Z \setminus \overline{\Omega}$ such that, for some constants $\alpha = \alpha(c,n),\beta = \beta(c,n)>0$, it holds that
\begin{enumerate}
    \item $\diam B_i \leq \alpha d( B_i, \overline{\Omega} )$ for $i \in I$;
    \item every set $D \subset Z \setminus \overline{\Omega}$ with $\diam(D) \leq \beta d( D, \overline{\Omega} )$ meets at most $(n+1)$ members of $( B_i )_{ i \in I }$.
\end{enumerate}
As in the proof of \cite[Theorem 5.2]{La:Sch:05}, we let $\delta = \beta/ (2(\beta+1))$ and consider the $1$-Lipschitz functions $\sigma_i \colon Z \setminus \overline{\Omega} \rightarrow \mathbb{R}$ where
\begin{align*}
    \sigma_i(x) = \max\left\{ 0, \delta d( \overline{\Omega}, B_i ) - d( B_i, x ) \right\}, \quad\text{for $x \in Z$.}
\end{align*}
For every $x \in Z \setminus \overline{\Omega}$, let $I_x = \left\{ i \in I \colon \sigma_i(x) > 0 \right\}$. We claim that  $1 \leq \sharp I_x \leq n+1$. Indeed, for every $i \in I_x$, let $x_i \in B_i$ such that $d( x_i, x ) < \delta d( B_i, \overline{\Omega} ) \leq \delta d( x_i, \overline{\Omega} )$. Let $D = \bigcup_{ i \in I_x } \{x_i\}$ and notice that
\begin{align*}
    \diam D \leq 2\delta \sup_{ i } d( x_i, \overline{\Omega} ) \leq 2\delta( \diam D + d( D, \overline{\Omega} ) ),
\end{align*}
so $\diam D \leq \beta d ( D, \overline{\Omega} )$. By (2) above, it holds that $D$ meets at most $(n+1)$ members of $( B_i )_{ i \in I }$. Therefore $1 \leq \sharp I_x \leq n+1$. 

We may now define $\overline{\sigma}(x) = \sum_{ i } \sigma_i(x) = \sum_{ i \in I_x } \sigma_i(x)$ and $\overline{\sigma}_i(x) = \sigma_i(x)/\overline{\sigma}(x)$ for every $x \in Z \setminus \overline{\Omega}$. By considering $z_i \in \overline{\Omega}$ such that $d( z_i, B_i ) < (2-\delta) d( \overline{\Omega}, B_i )$ for each $i \in I$, we define
\begin{align*}
    F(x) = \sum_{ i \in I_x } \overline{\sigma}_i(x) f(z_i), \quad x\in Z\setminus \overline\Omega, \quad f \in \LIP^{\infty}( \overline{\Omega}; \mathbb{V} ).
\end{align*}
The linear extension operator is defined by $E_{\mathbb V}\colon \LIP^{\infty}(\overline \Omega;\mathbb V)\to \LIP^{\infty}(Z;\mathbb V)$ where 
$$E_{\mathbb{V}}f=\begin{cases}
    f, & \text{in}\; \overline \Omega\text{, and}\\
    F, & \text{in} \; Z\setminus \overline \Omega.
\end{cases}$$
Indeed, the linearity is clear and the upper bound $\| E_{\mathbb{V}}f \|_{L^{\infty}(Z)} \leq (n+1)\|f\|_{L^{\infty}(\overline{\Omega})}$ is immediate. The validity of $\lip( E_{\mathbb{V}}f ) \leq C \lip(f)$ for a constant $C = C(n,\alpha,\beta)$ requires further work. Nevertheless, we leave the remaining standard details to the reader; similar estimates can be found in the proof of \cite[Theorem 5.2]{La:Sch:05}.
\end{proof}

\Cref{prop:Banach-valued_version} is a straightforward corollary of \Cref{thm:Main_inf.ext.} and \Cref{prop:real-to-Banach}.
\begin{proof}[Proof of \Cref{prop:Banach-valued_version}]
By using the notation from the proof of \Cref{prop:M=W_Banach-valued}, it follows that an extension operator $E_{\mathbb{V}} \colon W^{1,\infty}(\Omega; \mathbb{V}) \rightarrow W^{1,\infty}(Z; \mathbb{V})$ can be precomposed by an isometric embedding $I_v \colon W^{1,\infty}(\Omega) \rightarrow W^{1,\infty}(\Omega;\mathbb V)$ and postcomposed by the projection operator $$P_w \colon W^{1,\infty}(Z; \mathbb{V}) \rightarrow W^{1,\infty}(Z)$$ to obtain an extension operator $P_w \circ E_{\mathbb V} \circ I_v $ from $W^{1,\infty}(\Omega)$ into $W^{1,\infty}(Z)$.

Regarding the converse, by arguing as in the proof of \Cref{thm:Main_inf.ext.}, we have that $\LIP^{\infty}(\Omega)=W^{1,\infty}(\Omega)$, so the conclusion follows from \Cref{prop:real-to-Banach}.
\end{proof}

\section{Examples of infinitesimally doubling spaces}\label{sec:examples}

The subsequent \Cref{ex:example:collapsing,ex:example:collapsing:more,ex:CAT(0),ex:Lyt-Wen} show various geometric contexts where a metric measure space $Z$ is not doubling but still satisfies our assumptions (see also \Cref{ex:2}). In fact, the equality $\LIP^\infty(Z)=W^{1,\infty}(Z)$ holds isometrically for the first three families of examples.

  We start by constructing a large family of geodesic metric measure spaces $Z$, equipped with the $n$-dimensional Hausdorff measure $\mathcal{H}^n$, such that outside a single point, $Z$ is locally isometric to $\mathbb{R}^n$ and $W^{1,\infty}(Z) = \LIP^{\infty}(Z)$ isometrically. These are obtained by collapsing a \emph{continuum}, i.e. a non-empty compact and connected set, in $\mathbb{R}^n$ to a point. If the collapsing set is not a singleton, the reference measure is not doubling.
\begin{example}\label{ex:example:collapsing}
{\em Let $n \geq 2$ and consider any continuum $E \subset \mathbb{R}^{n}$. Denote
\begin{align*}
    d(x,y) \coloneqq \min\{ |x-y|, \mathrm{dist}(x,E) + \mathrm{dist}(E,y) \}
    \quad\text{for $x, y \in \mathbb{R}^{n}$}
\end{align*}
where $\mathrm{dist}(x,E) = \inf_{ y \in E } d( x, y )$. Notice that $d$ is nonnegative, finite, and satisfies the triangle inequality. Consider the quotient map $Q \colon \mathbb{R}^{n} \rightarrow Z$ collapsing $E$ to a point. We equip $Z$ with the quotient distance induced by $d$ and the $n$-dimensional Hausdorff measure $\mathcal{H}^n$ obtained from the quotient distance. Observe that $Q$ is $1$-Lipschitz in $\mathbb{R}^n$, measure-preserving in $\mathbb{R}^{n} \setminus E$, and that $\mathcal{H}^{n}( Q(E) ) = 0$. It is immediate that $Z$ is infinitesimally doubling.

We first prove that $W^{1,\infty}(Z) = \LIP^{\infty}(Z)$ isometrically. To this end, we fix a function $u\in\widehat N^{1,\infty}(Z)$ and show that $u$ is Lipschitz continuous and satisfies $\lip(u) = \mathscr{E}_{\infty}(u)$. This will show the claim.

It holds that $h = u \circ Q \in \widehat{N}^{1,\infty}(\mathbb{R}^n)$. Indeed, this follows because a curve family $\Gamma \subset \mathbb{R}^n$ is $\infty$-negligible if the image family $Q\Gamma$ is $\infty$-negligible. Since the Euclidean and essential metrics on $\mathbb{R}^n$ coincide, the definition of  $\widehat{N}^{1,\infty}(\mathbb{R}^n)$ implies that $h$ is Lipschitz.

Given that $h(z) = h(z')$ whenever $z, z' \in E$, it holds that
\begin{align*}
    | h(x) - h(y) |
    &=
    \min\left\{
        |h(x)-h(y)|,
        | h(x)-h(z) | + | h(z') - h(y) |
    \right\}
    \\
    &\leq
    \lip(h)
    \min\left\{
        | x-y |,
        | x-z | + | z'-y |
    \right\}
    \quad\text{for $x,y \in \mathbb{R}^n \setminus E$.}
\end{align*}
Taking the infimum over $z, z' \in E$ leads to the Lipschitz estimate $\lip(u) = \lip( u|_{ Z \setminus Q(E) } ) \leq \lip(h) = \mathscr{E}_{\infty}(h)$. Since $\mathscr{E}_{\infty}(h) \leq \mathscr{E}_{\infty}(u)$ and $\mathscr{E}_{\infty}(u) \leq \lip(u)$, it follows that $\lip( h ) = \lip( u ) = \mathscr{E}_{\infty}(u)$.

Secondly, we show that the reference measure in $Z$ is not doubling when $E$ contains more than two points. To this end, let $x_0 \in \mathbb{R}^n \setminus E$ and let $z_0 \in E$ be a nearest point to $x_0$. Let $\gamma \colon [0,|x_0-z_0|] \rightarrow \mathbb{R}^n$ be a unit speed Euclidean line segment joining $z_0$ to $x_0$. Then letting $\lambda = 3/2$, every $s\in (0,(2/3)|x_0-z_0|]$ satisfies
\begin{align*}
    \mathcal{H}^{n}\left( B\left( Q\left(\gamma\left( \lambda s \right) \right), s \right) \right)
    &=
    \mathcal{L}^{n}( B( \gamma(\lambda s ), s ) )
    =
    \omega_n s^n
    \quad\text{and}\quad
    \\
    \mathcal{H}^{n}\left( B\left( Q\left( \gamma\left( \lambda s \right) \right), 2s \right) \right)
    &\geq
    \omega_n s^n
    +
    \mathcal{L}^{n}\left( \left\{ x \in \mathbb{R}^n \colon 0 < \mathrm{dist}( x, E ) < s/2 \right\} \right).
\end{align*}
To finish, the failure of the doubling property follows if there exist $c > 0$ and $d \in (0,n)$ such that
\begin{align*}
    \liminf_{ s \rightarrow 0^{+} }
    \frac{ \mathcal{L}^{n}( \left\{ x \in \mathbb{R}^n \colon 0 < \mathrm{dist}( x, E ) < s/2 \right\} ) }{ s^d }
    \geq
    c.
\end{align*}
When $\mathcal{L}^{n}(E) > 0$, this follows with the value $d=1$ from a corollary of the Brunn--Minkowski inequality \cite[3.2.42]{Fed:69}. When $\mathcal{L}^{n}(E) = 0$, it follows that $E = \partial E$ and thus it suffices to derive a lower bound for the $1$-dimensional lower Minkowski content of $E$; see \cite[Section 5.5]{Mattila:95} for the definition and basic properties. Since $E$ is a continuum, up to a constant, the $1$-dimensional lower Minkowski content is bounded from below by the $1$-dimensional Hausdorff measure of $E$ and thus by the positive diameter of $E$. The desired lower bound follows in both cases.}
\end{example}
\Cref{ex:example:collapsing} can be modified by collapsing a collection of continua instead of a single one. To this end, consider a collection of pairwise disjoint continua $\{ E_k \}_{ k \in \mathcal{I} }$ in $\mathbb{R}^n$ for a non-empty index set $\mathcal{I}$. We impose the condition that $K = \bigcup_{ k \in \mathcal{I} } E_k$ is compact. An example to keep in mind is when $K = \mathcal{I} \times [0,1]$ for a Cantor set $\mathcal{I} \subset \mathbb{R}^{n-1}$ in which case the sets $\{  E_k \}_{ k \in \mathcal{I} }$ index the connected components of $K$.
\begin{example}\label{ex:example:collapsing:more}
{\em We consider the auxiliary function
\begin{align*}
    D( x, y )
    \coloneqq
    \inf_{i\in\mathcal I}\{ |x-y|, \mathrm{dist}(x,E_i) + \mathrm{dist}(E_i,y) \}
    \quad\text{for $x, y \in \mathbb{R}^{n}$}
\end{align*}
 and let
\begin{align*}
    d( x, y )
    =
    \inf \left\{\sum_{ j } D( x_j, x_{j+1} ):\, \text{$( x_j )_{j \geq 1}^{l}$ is a finite chain joining $x$ to $y$}\right\}
    \quad\text{for $x, y \in \mathbb{R}^n$.}
\end{align*}
If $Q \colon \mathbb{R}^n \rightarrow Z$ is the quotient map collapsing each $E_j$ to a point, then $d$ defines a unique distance in $Z$ so that $Q$ is $1$-Lipschitz in $\mathbb{R}^n$ and a local isometry in $\mathbb{R}^n \setminus K$. When we equip $Z$ with the Hausdorff measure $\mathcal{H}^n$, the map $Q$ becomes measure-preserving in $\mathbb{R}^n \setminus K$. Again, it holds that if $u \in \widehat{N}^{1,\infty}(Z)$, then $h = u \circ Q$ is Lipschitz. In particular, $u$ is continuous and $h|_{ E_i }$ is constant for every $i \in \mathcal{I}$. The latter observation implies
\begin{align*}
    |h(x)-h(y)| \leq \lip(h)D(x,y) \quad\text{for $x, y \in \mathbb{R}^n \setminus K$,}
\end{align*}
and thus $\lip(u) \leq \lip(h)$. By arguing as in \Cref{ex:example:collapsing}, it follows that $\lip(u) = \lip(h) = \mathscr{E}_{\infty}(u)$, and therefore $\LIP^\infty(Z)=W^{1,\infty}(Z)$ isometrically. Since $Q$ is $1$-Lipschitz, we conclude that $Z$ is $n$-rectifiable and thus infinitesimally doubling.

When $\mathcal{L}^{n}( K ) = 0$ and at least one of the components of $K$ has positive diameter, the measure $\mathcal{H}^n$ is not doubling. This follows from a minor modification of the argument in \Cref{ex:example:collapsing}.}
\end{example}
In \Cref{ex:example:collapsing:more}, the space $\LIP^{\infty}(Z)$ can be isometrically identified with the subspace $\mathcal{F}=\{h\in \LIP^{\infty}(\mathbb{R}^n):\, h|_{ E_i }\; \text{is a constant for every $i \in \mathcal{I}$}\}$. When $\mathcal{I}$ is finite, the constants $h|_{ E_i }$ can be prescribed independently. When the points in $Q(K)$ are not isolated, this is not possible. We give two examples. If $K = \mathcal{I} \times [0,1]$ for a Cantor set $\mathcal{I} \subset [0,1]$, then each component $C$ of $K$ is an accumulation point of the connected components of $K \setminus C$. Similarly, there is a compact set $K \subset \mathbb{R}^2$ such that every connected component of $K$ is a point or a Euclidean circle (i.e. the boundary of a disc) and every point of $K$ is an accumulation point of the circle components, cf. \cite[Example, p. 1308-1309]{Younsi:16}.

 Let us relate \Cref{ex:example:collapsing:more} with AMLEs. To this end, consider a bounded domain $\Omega \subset Z$ such that $Q(K) \subset \Omega$. Then an AMLE $u \colon Z \rightarrow \mathbb{R}$ of $g \colon Z \setminus \Omega \rightarrow \mathbb{R}$ is such that $u \circ Q$ is constant in bounded components of $\mathbb{R}^n \setminus K$ separated from $Q^{-1}( Z \setminus \Omega )$. For instance, in the latter example (the one based on \cite[Example, p. 1308-1309]{Younsi:16}), the function $u \circ Q$ is $\infty$-harmonic in $\Omega' = Q^{-1}( \Omega ) \setminus K$ and constant in the bounded components of $\mathbb{R}^2 \setminus K$.

We consider three further examples of metric measure spaces satisfying our assumptions but not necessarily having a (locally) doubling measure. The first involves simplicial complexes, the second curvature upper bounds in the sense of Alexandrov, and the third one solutions of the Plateau problem in metric spaces.
\begin{example}\label{ex:CAT(0)}
{\em We consider simplicial complexes and refer to \cite[Section 7]{Cavallucci:Sambusetti:2022} for the relevant definitions.

Let $Z$ be a connected and locally compact simplicial complex $Z$. Here $Z$ is equipped with the standard length metric where the length on each simplex is measured using the Euclidean length. We equip $Z$ with the measure $\mu$ that restricts to the $j$-dimensional Hausdorff measure in the interior of any $j$-dimensional simplex. Since every point of $Z$ contains an open neighbourhood with a finite number of simplices, it follows that $\mu$ is locally finite. Using this fact, the measure $\mu$ being infinitesimally doubling also follows. Note that $Z$ is proper as a locally compact length space, so $\mu$ is finite on bounded sets.

We claim that $W^{1,\infty}(Z) = \LIP^{\infty}(Z)$ isometrically. By \Cref{prop:sharpness} and recalling that every $\infty$-negligible family is contained in some $\Gamma^{+}_{N}$, for a Borel set $N \subset Z$ with $\mu(N) = 0$, we only need to prove that if $x, y \in Z$, $\epsilon > 0$, then there exists a curve $\theta$ joining $x$ to $y$, having zero length in $N$, and having length $\ell( \theta )\leq d(x,y) + \epsilon$.

By construction of the distance on $Z$, it follows that every $x, y \in Z$ can be joined with a curve $\gamma$ of length $d(x,y) + \epsilon/2$ consisting of $M$ geodesic segments $\{\gamma_i\}^{M}_{i=1} $ ($M$ possibly depending on $x$ and $y$ and $\epsilon >0$), each having length $\ell_i$ and being contained in a simplex $\Delta_i$ of $Z$, with the possible exception of the end points, for $1 \leq i \leq M$. If $\Delta_i$ is one-dimensional, we set $\theta_i = \gamma_i$ as no perturbation is needed. If $\Delta_i$ is at least two-dimensional, there exists a curve $\theta_i$ with the same endpoints as $\gamma_i$ in $\overline{\Delta}_i$, of length $\ell_i + \epsilon/(2M)$, and having zero length in $N$. Concatenating $\theta_i$ defines a curve $\theta$ with the desired properties.

The measure $\mu$ is not doubling up to any scale. Indeed, suppose that $\mu$ were doubling with a doubling constant $C$ and up to scale $r_1$. If $x_0 \in Z$ is a vertex, then $\mu( B(x_0,r) ) \geq \mathcal H^0(x_0)=1$ and there exist $c > 0$ and $r_0 \in (0,r_1/2)$ with the following property: for every $r \in (0,r_0)$, there exists $y \in B( x_0, 2r ) \setminus B( x_0, r )$ such that $\mu( B( y, r ) ) \leq c r$. It follows that
\begin{align*}
    1 \leq \mu( B( x_0, 2r ) ) \leq \mu( B( y, 4r ) ) \leq C^2 \mu( B( y, r ) ) \leq C^2 c r
    \quad\text{for $r \in (0,r_0)$.}
\end{align*}
This leads to a contradiction as $r \rightarrow 0^{+}$.

Notice that if a finite simplicial complex $Z$, say, is realized as a subset of $\mathbb{R}^K$ for some $K \in \mathbb{N}$, then using the ambient Euclidean distance and the measure $\mu$ above yields an example of a metric measure space with an infinitesimally doubling measure and which is very $\infty$-thick $(C,R)$-quasiconvex for some $C > 1$ and $R > 0$. The discussion above implies that the essential distance on a connected component coincides with the intrinsic distance.}
\end{example}
\begin{remark}
\Cref{ex:CAT(0)} relates to the work by Lytchak and Nagano in \cite{Lyt:Nag}. They proved that every locally compact locally geodesically complete metric space $Z$ with a curvature upper bound in the sense of Alexandrov (GCBA space) has a natural measure $\mu$ that is infinitesimally doubling. The measure consists of a sum $\mu = \sum_{ j = 0 }^{ k } \mathcal{H}^{j}\llcorner{ Z_j }$, where $\mathcal{H}^j$ is the $j$-dimensional Hausdorff measure. Here $Z_0, \dots, Z_k$ is a stratification of $Z$ for which $k$ is equal to the topological dimension of $Z$. The key properties of $\mu$ include that, for every $1 \leq j \leq k$, there exists an open set $U_j \subset Z$ contained in $Z_j$ such that the Hausdorff dimension of $\overline{Z}_j \setminus U_j$ is at most $(j-1)$, and that every point in $U_j$ has an open neighbourhood in $U_j$ bi-Lipschitz homeomorphic to an open ball in $\mathbb{R}^j$. Lytchak and Nagano also prove that $\mu$ is positive and finite on each open and relatively compact subset of $Z$. It is immediate that every point in $\bigcup_{j = 1 }^{ k } U_j$ has an open neighbourhood $U$ for which $\LIP^{\infty}(U) = W^{1,\infty}(U)$ holds. It would be interesting to know if $\LIP^{\infty}(Z) = W^{1,\infty}(Z)$ holds.
\end{remark}
\begin{example}\label{ex:Lyt-Wen}
{\em Metric measure spaces which are infinitesimally doubling but not necessarily doubling naturally occur in the context of the Plateau problem as formulated by Lytchak and Wenger \cite{Lyt:Wen:17:areamini,Lyt:Wen:17:energyarea,Lyt:Wen:18:intrinsic,Lyt:Wen:18:CAT}. In fact, under mild assumptions on a metric space $X$, if $u \colon \overline{\mathbb{D}} \rightarrow X$ is a solution for the Plateau problem, there exists an infinitesimally doubling metric measure space $Z$ with $\LIP^{\infty}(Z) = W^{1,\infty}(Z)$ homeomorphic to the closed disc $\overline{\mathbb{D}}$, a Sobolev map $v \colon \overline{\mathbb{D}} \to Z$, and a $1$-Lipschitz map $P \colon Z \to X$ for which $u = P \circ v$. This follows from \cite[Theorem 1.9]{Creutz_Soultanis}. 

The properties of $Z$ are closely related to the structure of the minimizer $u$, cf. \cite{Lyt:Wen:18:intrinsic,Lyt:Wen:18:CAT}. The two-dimensional Hausdorff measure $Z$ is not doubling in some situations. For instance, \cite[Example 11.3]{Lyt:Wen:18:CAT} considers a similar construction as in \Cref{ex:example:collapsing} for a disc. See also \cite[Example 6.4]{Creutz_Soultanis}.}
\end{example}

\newcommand{\etalchar}[1]{$^{#1}$}

\end{document}